\documentclass{article}

\usepackage[latin1]{inputenc}
\usepackage[T1]{fontenc}
\usepackage[english]{babel}
\usepackage[top=3cm, bottom=3cm, left=3cm, right =3cm]{geometry}
\usepackage{amsthm}  
\usepackage{amssymb}
\usepackage{eufrak}
\usepackage{MnSymbol}
\usepackage{xcolor}
\usepackage{enumitem}
\usepackage{multirow}
\author{Benjamin \textsc{Druart}\thanks{Université de Grenoble I, Département de Mathématiques, Institut Fourier, UMR 5582 du CNRS, 38402 Saint-Martin d'Hères Cedex, France. email : Benjamin.Druart@ujf-grenoble.fr}
\thanks{The research leading to these results has recieved funding from the European Research Council under the European Community's Seventh Framework Programme FP7/2007-2013 Grant Agreement no. 278722 and was partially supported by ANR-13-BS01-0006-01 ValCoMo}
}
\title{Definable subgroups in $SL_2$ over a $p$-adically closed field}

\newtheorem{defi}{Definition}
\newtheorem{fact}{Fact}
\newtheorem{prop}[defi]{Proposition}
\newtheorem{cor}[defi]{Corollary}
\newtheorem{lem}[defi]{Lemma}
\newtheorem{thm}{Theorem}
\newtheorem{conj}{Conjecture}

\theoremstyle{definition}
\newtheorem*{rmk}{Remark}

\newcommand{\Qp}{\mathbb{Q}_p}
\newcommand{\Zp}{\mathbb{Z}_p}
\newcommand{\Z}{\mathbb{Z}}

\newcommand{\Fp}{\mathbb{F}_p}
\newcommand{\OO}{\mathcal{O}}

\newcommand*{\slfrac}[2]%
{\ensuremath{%
    #1/\!\raisebox{-.65ex}{\ensuremath{#2}}}}
    
\newenvironment{lembis}[1]
  {\renewcommand{\thedefi}{\ref{#1} bis}%
   \addtocounter{defi}{-1}%
   \begin{lem}}
  {\end{lem}}
  
  \newenvironment{propbis}[1]
  {\renewcommand{\thedefi}{\ref{#1} bis}%
   \addtocounter{defi}{-1}%
   \begin{prop}}
  {\end{prop}}
  
  \newenvironment{thmbis}[1]
  {\renewcommand{\thethm}{\ref{#1} bis}%
   \addtocounter{thm}{-1}%
   \begin{thm}}
  {\end{thm}}
  
\begin{document}

\maketitle

%\tableofcontents
%\begin{abstract}
%We show that there exist a finite number of Cartan subgroups up to conjugacy in $SL_2(\Qp)$ and we describe all of them. We show that the Cartan subgroup consisting of all diagonal matrices is generous and it is the only one up to conjugacy.
%\end{abstract}
%\textit{Keywords } p-adic field ; Cartan subgroup ; generosity
%\\ \textit{MSC2010 } 20G25 ; 20E34 ; 11E57

\bigskip
\begin{abstract}
The aim of this paper is to describe all definable subgroups of $SL_2(K)$, for $K$ a $p$-adically closed field. We begin by giving some "frame subgroups" which contain all nilpotent or solvable subgroups of $SL_2(K)$. A complete description is givien for $K$ a $p$-adically closed field, some results are generalizable to $K$ a field elementarily equivalent to a finite extension of $\Qp$ and an almost complete description is given for $\Qp^{an}$. We give also some indications about genericity and generosity in $SL_2(K)$.
\end{abstract}

\section*{Introduction}
Studying definable groups and their properties is an important part of Model Theory. 
Frequently, general hypotheses of Model Theory, such as stability or NIP, are used in order to specify the context of the analysis. In this article we adopt a different viewpoint and concentrate on the study of a specific  class of groups, namely $SL_2(K)$ where $K$ is a $p$-adically closed field. 

There are several reasons why we adopt this approach, the first of which is the exceptional algebraic and definable behaviour of linear groups of small dimension, in particular of $SL_2$. A good example of this phenomenon is in \cite[9.8]{cartan}, where the case $SL_2$ is treated separately and which is also a source of inspiration for the current article. In \cite{SL2R}, Gismatulin, Penazzi and Pillay work in $SL_2(\mathbb{R})$ and its type space to find a counterexample to a conjecture of Newelski. Although we do not find a similar result for $SL_2(\Qp)$, we hope that  this study will contribute to the problem. 
We are also motivated by finding concrete examples of definable groups in unstable but model-theoretically well behaved structures. 
%Many people adopt a combinatorics point of vue : which configuration is possible in which context (existence of chain condition  for stable, $\omega$-stable, $\mathfrak{M}_c$, or NIP groups ...). Here we work on a different way. We want to describe all definable subgroups of one specific group, in fonction of theirs algebraic properties. Our methode give then examples of definable groups in different context. 

The starting point of our study is Cartan subgroups. They have been extensively studied in the context of groups of finite Morley rank \cite{FrecJal}\cite{Frec}, and for definable groups  in $o$-minimal structures \cite{cartan}. Here they will be used to describe definable subgroups in $SL_2(K)$. 
Motivated by the work in \cite{cartan}, Jaligot raised the following question: 
Does a definable group in some specific context have a finite number of conjugacy classes of Cartan subgroups ? Here, we answer this question for $SL_2(K)$ (Theorem \ref{Cartan}), where $K$ varies over a large class of fields with few contraints that will be detailed in section 1. This class includes algebraically closed fields, real closed fields, $p$-adically closed fields, pseudoreal closed fields, pseudo $p$-adically closed fields, both with bounded Galois group. We have thus a large overview of model-theoretical contexts (stable, $NIP$, simple, $NTP_2$ ...). 
%Jaligot ask the question to know if a definable group in some specific context has  a finite number of conjugacy class of Cartan subgroups. Baro, Jaligot and Otero answer affirmatively it for definable groups in $o$-minimal structure \cite{cartan}. The question is open for other contexts and especially for $p$-adically closed fields. Here we answer the question for $SL_2(K)$ (Theorem \ref{Cartan}). 
%The demonstration involves elementary algebaic technics, then it works for many deffirent contexts : $K$ can be an  

Our study shows the existence of some "frame" subgroups (they are maximal nilpotent or solvable subgroups - Corollary \ref{nilpotent} and Proposition \ref{solvable}), which contain all the nilpotent or solvable subgroups. 
The description of "frame" subgroups is uniform for different fields. A finer description depends on the model-theoretical nature of the underlying field. The cases of algebraically closed and real closed fields can be treated very quickly, the $p$-adic case is the first non-trivial case. 

We then give a complete description of definable subgroups of $SL_2(K)$, for $K$ a $p$-adically closed field. 
Theorem \ref{L3} gives us a full description of definable subgroups of anisotropic tori of $SL_2(K)$: they form an infinite descending chain of definable subgroups. 

In the non solvable case, topological properties determine the nature of the subgroups in question. 
In theorem \ref{unbounded}, we show that if $H$ is a non solvable and unbounded definable subgroup of $SL_2(K)$, then $H=SL_2(K)$. 

Finally a table sums up informations about definable subgroups of $SL_2(K)$. Some propositions for $p$-adically closed fields are true for some henselian fields of characteristic 0, in particular finite extensions of $p$-adically closed fields.

The final section  discusses notions of genericity, generosity and which definable subgroups of $SL_2(K)$ are generous (Corollary \ref{generous}). 

A natural question is whether $SL_2(K)$ has the same definable subgroups if we extend the language. A natural way to answer this question is to explore $p$-minimal expansions of a $p$-adically closed field. $p$-minimality is the $p$-adic analogue of $o$-minimality, it defines  a class of expansions whose definable sets stay "close" to the definable sets in field language. Unfortunately, $p$-minimality is not well kown, especially we ignore if there exists a dimension theory compatible with topology. Thus we restrict ourselves to $\Qp^{an}$ the expansion of $\Qp$ where we adjoin all "restricted analytic functions". This latter is an important and extensively studied example of $p$-minimal expansion of $\Qp$.

%PRELIMINARIES
\section{Preliminaries in Model Theory}
In this section we review notions from model theory.  For further details ,we refer to \cite{hodges}. 

A \emph{structure} $M$ is a non empty set together with some distinguished fonctions, relations and constants, these form the signature. In order to define subets of $M^n$, one introduce symbols, one for each element of the signature. The symbols together with the equality give the \emph{language}. For example a group is naturally considered in the language $\mathcal{L}_G=\{\cdot,\  ^{-1}, e\}$ for the composition law, the inversion and the neutral element and a field in $\mathcal{L}_R=\{+,-,\cdot, 0,1\}$. Definitions use \emph{first-order formulas} which are  coherent sequences formed with symbols from the language, variables, logical connectives and quantifiers. Only, quantification on the elements of $M$ is allowed, this excludes quantifying a part of $M$ or integers.

Two structures are \emph{elementarily equivalent} if they satisfy exactly the same first order formulas without free variable (sentences), then their \emph{theory} is the set of all formulas verified by the structures. 
For a structure $M$, a set $X\subseteq M^n$ is said \emph{definable} if there exists a formula $\varphi(\bar{x},\bar{a})$ with some parameter $\bar{a}\subseteq M$, such that elements of X are exactly the elements of $M^n$ verifying the formula $\varphi$, then we note $X=\varphi(M)$. A group is definable in $M$ if its set of elements is definable and the graph of the composition law is definable. 
A structure $N$ is \emph{interpetable} if there exits a definable set $X$ of $M^n$, a definable equivalence relation $E$ in $M^{2n}$ and a bijection $f$ from $\slfrac{X}{E}$ to $N$, such that the inverse image of fonctions and relations of $N$ by $f$ are definable in $X$.

For example, $SL_2(K)$ can be seen as a structure of group in the group language $\mathcal{L}_G=\{\cdot,^{-1}, e\}$, it is also a definable group in the structure of field $K$ in the language $\mathcal{L}_R=\{+,-,\cdot , 0,1\}$. The question in this paper is : what are the definable subgroups of $SL_2(K)$ in the field language $\mathcal{L}_R$ ? 

An \emph{expansion} of a structure $M$ is the structure where we add some symbols in the language. For example $\Qp^{an}$ is the structure of $p$-adic number studied with one symbol for every "restricted analytic functions". 

\bigskip

\textbf{Notation: }If $K$ is a field, we denote by $K^+$ and $K^{\times}$ respectively the additive and the multiplicative groups of $K$, and $\widetilde{K}^{alg}$ will be the algebraic closure of $K$. For a group $G$, we denote by $(G)^n$ the set of $n^{th}$ powers from $G$.
The $n^{th}$ cartesian power of a set $K$ will be noted $K^n$.

Definable means definable with parameters.

%CARTAN SUBGROUPS
\section{Cartan subgroups in $SL_2(K)$}

In this section, we work in a very general setting, $K$ will be an infinite field such that $\slfrac{K^{\times}}{(K^{\times})^2}$ is  finite and $char{K}\neq 2$. These properties are preserved if we replace $K$ by an elementarily equivalent field. They are verified by algebraically closed fields, real closed fields, $p$-adically closed fields, finite extensions of $\Qp$, pseudoreally closed fields, pseudo-$p$-adically closed fields with bounded Galois group. This last hypothesis means that for all $n$, $K$ has finitely many extensions of degree $n$. We know by \cite{Samaria} that pseudoreal closed fields and pseudo $p$-adically closed fields with such hypothesis are $NTP_2$. First we are studying Cartan subgroups of $SL_2(K)$, we are showing that there exists a finite number of them up to conjugacy (Theorem \ref{Cartan}). Since in general, in $SL_2$ every element is unipotent or semisimple, we conclude in Corollary \ref{union} that a finite number of Cartan subgroups suffice to describe all semisimple groups. After we define some "frame" groups (Corollary \ref{nilpotent} and Proposition \ref{solvable}), which are definable in the pure group and such that they contain every  nilpotent or solvable subgroups.

The next definition is due to Chevalley, it defines Cartan subgroups in an abstract group. We show, as one would expect it, that in our context Cartan subgroups are maximal tori.

%We note $v_p : \Qp\longrightarrow \Z\bigcup\{+\infty\}$ the $p$-adic-valuation, and $ac : \Qp^{\times} \longrightarrow \mathbb{F}_p$ the angular component defined by $ac(x)=res(p^{-v_p(x)}x)$ where $res: \Qp\longrightarrow \mathbb{F}_p$ is the residue map. 

%With these notations, if $p\neq 2$, an element $x\in \Qp^{\times}$ is a square if and only if $v_p(x)$ is even and $ac(x)$ is a square in $\mathbb{F}_p$. For $p=2$, an element $x\in\mathbb{Q}_2$ can be written $x=2^nu$ with $n\in \Z$ and $u\in \mathbb{Z}_2^{\times}$, then $x$ is a square if $n$ is even and $u\equiv 1 (\mbox{mod } 8)$ \cite{JPS}.

%\begin{fact}[\cite{JPS}]\label{Serre}
%If $p\neq2$, the group $\Qp^{\times} / (\Qp^{\times})^2$ is isomorphic to $\Z/2\Z \times \Z/2\Z$, it has for representatives $\{1, u, p, up\}$, where $u\in \Zp^{\times}$ is such that $ac(u)$ is not a square in $\mathbb{F}_p$

%The group $\mathbb{Q}_2^{\times}/(\mathbb{Q}_2^{\times})^2$ is isomorphic to $\Z/2\Z \times \Z/2\Z \times \Z/2\Z$, it has for representatives $\{\pm1,\pm 2, \pm 5, \pm10\}$. 
%\end{fact}

%\subsection{Cartan subgroups}
\begin{defi}
Let $G$ be a group. A subgroup $C$ is a Cartan subgroup if : 
\begin{enumerate}
	\item $C$ is a maximal nilpotent subgroup ;
	\item every subgroup of finite index $X\leq C$, is of finite index in its normalizer $N_G(X)$.
\end{enumerate}
\end{defi}

\begin{rmk}
If $G$ is infinite, then every Cartan subgroup is infinite. Indeed if $C$ is a finite Cartan subgroup of $G$ then $\{1\}$ is a subgroup of finite index in $C$ but is of infinite index in $N_G(\{1\})=G$.
\end{rmk}

For any $\delta\in K^{\times}\backslash (K^{\times})^2$, we put : 
\begin{align}
Q_1 &=\left\{\left(\begin{array}{cc}a & 0 \\0 & a^{-1}\end{array}\right)\in SL_2(K)\mid a\in K^{\times}\right\} \nonumber\\ 
Q_{\delta}&=\left\{\left(\begin{array}{cc}a & b \\b\delta & a\end{array}\right)\in SL_2(K) \mid a,b\in K \mbox{ and } a^2-b^2 \delta=1\right\} \nonumber
\end{align}

\begin{lem}\label{centre}
\begin{align} 
 \forall x\in Q_1\backslash \{I,-I\} \quad C_{SL_2(K)}(x)=Q_1\nonumber
\\ \forall x\in Q_{\delta}\backslash \{I,-I\} \quad C_{SL_2(K)}(x)=Q_{\delta} \nonumber
\end{align}
\end{lem}

The checking of these equalities is left to the reader.

\begin{prop}
The groups $Q_1$ and $Q_{\delta}$ are Cartan subgroups of $SL_2(K)$ 
\end{prop}

\begin{proof}
The group $Q_1$ is abelian and the normalizer of $Q_1$ is :
$$N_{SL_2(K)}(Q_1)=Q_1\cdot<\omega> \quad\mbox{ where } \quad\omega =\left(\begin{array}{cc}0 & 1 \\-1 & 0\end{array}\right)$$
For $X$ a subgroup of $Q_1$, if $g\in N_{SL_2(K)}(X)$ and $x\in X$, then, using lemma \ref{centre}
$$Q_1 = C_{SL_2(K)}(x)=C_{SL_2(K)}(x^g)=C_{SL_2(K)}(x)^g=Q_1^g$$ 
It follows that $N_{SL_2(K)}(X)\leq N_{SL_2(K)}(Q_1)=Q_1\cdot <\omega>$, as for $t\in Q_1$, $t^\omega=\omega^{-1}t\omega=t^{-1}$, we have $N_{SL_2(K)}(X)=Q_1\cdot<\omega>$ and if $X$ of finite index $k$ in $Q_1$, then $X$ is of index $2k$ in $N_{SL_2(K)}(X)$.

If we denote $\Gamma_i$ the descending central series of $N_{SL_2(K)}(Q_1)$, we can see that for $t$ in $Q_1$, $[\omega , t]=t^2$ and so we have :
$$\Gamma_0=N_{SL_2(K)}(Q_1)$$ 
$$\Gamma_1=[N_{SL_2(K)}(Q_1),N_{SL_2(K)}(Q_1)]=Q_1^2$$ 
$$\Gamma_i=[N_{SL_2(K)}(Q_1),\Gamma_{i-1}]=Q_1^{2^i}$$ 
Observing that $Q_1\cong K^{\times}$, we can conclude that $\Gamma_i$ never reaches $\{I\}$ because $Q_1$ is infinite and $Q_1^{2^i}$ is of finite index in $Q_1^{2^{i-1}}$ (indeed $\slfrac{Q_1^{2^{i-1}}}{Q_1^{2^i}}$ is in bijection with $\slfrac{K^{\times}}{(K^{\times})^2}$). Thus $N_{SL_2(K)}(Q_1)$ is not nilpotent.
By the normalizer condition for nilpotent groups, if $Q_1$ is properly contained in a nilpotent group $C$, then $Q_1 < N_C(Q_1) \leq C$, here $N_C(Q_1)=Q_1\cdot<\omega>$ which is not nilpotent, a contradiction. It finishes the proof that $Q_1$ is a Cartan subgroup.

For $\delta \in K^{\times}\backslash (K^{\times})^2$, the group $Q_{\delta}$ is abelian. Since for all subgroups $X$ of $Q_{\delta}$ not contained in $Z(SL_2(K))$, $C_{SL_2(K)}(X)=Q_{\delta}$, it follows that $N_{SL_2(K)}(X)=N_{SL_2(K)}(Q_{\delta})=Q_{\delta}$, and if $X$ is of finite index in $Q_{\delta}$ then $X$ is of finite index in its normalizer.  By the normalizer condition for nilpotent groups, $Q_{\delta}$ is maximal nilpotent. 
\end{proof}

\begin{prop}\label {tr}
\begin{enumerate}
\item $Q_1^{SL_2(K)} = \left\{ A\in SL_2(K)\mid tr(A)^2-4 \in (K^{\times})^2 \right\} \bigcup \left\{I,-I\right\}$

\item For any $\delta \in K^{\times}\backslash(K^{\times})^2$, there exist $n\in\mathbb{N}$ and $\mu_1, ... ,\mu_n\in GL_2(K)$ such that $$\bigcup _{i=1}^{n}Q_{\delta}^{\mu_i\cdot SL_2(K)} = \left\{ A\in SL_2(K)\mid tr(A)^2-4 \in \delta\cdot(K^{\times})^2 \right\} \bigcup \left\{I,-I\right\}$$
\end{enumerate}
\end{prop}

If $A$ is in $Q_1^{SL_2(K)}$ or in one of the $Q_{\delta}^{\mu_i\cdot SL_2(K)}$, $A$ is said to be a \textit{semi-simple} element, it means that there exists a field extension $K'/K$ such that $A$ is diagonalisable in $SL_2(K')$. A commutative subgroup  such that all elements are semi-simple is a \textit{torus}. $Q_1$ is a \textit{split torus} and the $Q_{\delta}$ are \textit{anisotropic tori}. If $K$ is an algebraically closed field, there is naturally no non trivial anisotropic torus.

We put : 
$$ U=\left\{\left(\begin{array}{cc}1 & u \\0 & 1\end{array}\right)\mid u\in K\right\}\bigcup\left\{\left(\begin{array}{cc}-1 & u \\0 & -1\end{array}\right)\mid u\in K\right\} \mbox{ and } U^+=\left\{\left(\begin{array}{cc}1 & u \\0 & 1\end{array}\right)\mid u\in K\right\}$$
If $A\in SL_2(K)$ satisfies $tr(A)^2-4=0$, then either $tr(A)=2$ or $tr(A)=-2$, and $A$ is a conjugate of an element of $U$. In this case, $A$ is said \textit{unipotent}.  It follows, from Proposition \ref{tr} : 

\begin{cor}
We have the following partition : 
$$SL_2(K)\backslash \{I,-I\} = (U\backslash \{I,-I\}) ^{SL_2(K)}\sqcup (Q_1\backslash \{I,-I\})^{SL_2(K)} \sqcup \bigsqcup_{\delta \in K^{\times}/(K^{\times})^2}\bigcup_{i=1}^{n}(Q_{\delta}^{\mu_i} \backslash \{I,-I\})^{SL_2(K)} \label{union}$$
\end{cor}

\begin{rmk}
If $\delta$ and $\delta'$ in $K^{\times}$ are in the same coset of $(K^{\times})^2$, then, by Proposition \ref{tr}, if $x'\in Q_{\delta'}^{\mu'}$ with $\mu'\in GL_2(K)$, then there exist $x\in Q_{\delta}$, $\mu\in GL_2(K)$ and $g\in SL_2(K)$, such that $x'=x^{\mu\cdot g}$, thus, by lemma \ref{centre}, $Q_{\delta'}=C_{SL_2(K)}(x')=C_{SL_2(K)}(x)^{\mu\cdot g}=Q_{\delta}^{\mu \cdot g}$.
Therefore in Corollary \ref{union} to speak of $Q_{\delta}$ up to conjugacy for $\delta \in \slfrac{K^{\times}}{(K^{\times})^2}$ makes sense.
\end{rmk}

\begin{proof}[Proof of Proposition \ref{tr}]
$\bullet$ If $A\in Q_1^{SL_2(K)}$, then there exists $P\in SL_2(K)$ such that 
$$A=P \left(\begin{array}{cc}a & 0 \\0 & a^{-1}\end{array}\right)P^{-1}$$
with $a\in K^{\times}$. We have $tr(A)=a+a^{-1}$, so $tr(A)^2-4=(a+a^{-1})^2-4=(a-a^{-1})^2$ and $tr(A)^2-4\in (K^{\times})^2$.

Conversely , let $A$ be in $SL_2(K)$ with $tr(A)^2-4$ a square. The characteristic polynomial is $\chi_A(X)=X^2-tr(A)X+1$ and its discriminant is $\Delta=tr(A)^2-4\in (K^{\times})^2$, so $\chi_A$ has two distinct roots in $K$ and $A$ is diagonalizable in $GL_2(K)$. There is $P\in GL_2(K)$, and  $D\in SL_2(K)$ diagonal such that $A=PDP^{-1}$. If
$$P=\left(\begin{array}{cc}\alpha & \beta \\\gamma & \delta \end{array}\right)$$ 
we put 
$$\tilde{P}=\left(\begin{array}{cc}\frac{\alpha}{det(P)} & \beta \\\frac{\gamma}{det(P)} & \delta\end{array}\right)$$ 
and we have $\tilde{P}\in SL_2(K)$ and $A=\tilde{P}D\tilde{P}^{-1} \in Q_1^{SL_2(K)}$.

$\bullet$
If $A$ is in $Q_{\delta}^{\mu\cdot SL_2(K)}\backslash \{I,-I\}$ with $\mu \in GL_2(K)$, then $tr(A)=2a$ and there exists $b\neq 0$ such that $a^2-b^2\delta=1$. 
So $tr(A)^2-4 = 4a^2-4=4(b^2\delta +1)-4=(2b)^2\delta \in \delta\cdot(K^{\times})^2$

Conversely we proceed as in the real case and the root $i\in\mathbb{C}$. The discriminant of $\chi_A$, $\Delta=tr(A)^2-4$ is a square in $K(\sqrt{\delta})$, and the characteristic polynomial $\chi _A$ has two roots in $K(\sqrt{\delta})$ : $\lambda_1 = a + b \sqrt{\delta}$ and $\lambda_2=a -b \sqrt{\delta}$ (with $a, b \in K$). For the two eigenvalues $\lambda_1$ and $\lambda_2$, A has eigenvectors : 
$$v_1=\left(\begin{array}{c}x+y\sqrt{\delta} \\x'+y'\sqrt{\delta}\end{array}\right) \quad \mbox{ and } \quad v_2=\left(\begin{array}{c}x-y\sqrt{\delta} \\x'-y'\sqrt{\delta}\end{array}\right)$$
 
In the basis $\left\{\left(x, x'\right),\left(y, y'\right)\right\}$, the matrix A can be written : 
$$\left(\begin{array}{cc}a & b \\b\delta & a\end{array}\right)$$ 
We can conclude that there exists $P\in GL_2(K)$ such that :
$$A=P\left(\begin{array}{cc}a & b \\b\delta & a\end{array}\right)P^{-1}$$
We proved that $Q_{\delta}^{GL_2(K)}= \left\{ A\in SL_2(K)\mid tr(A)^2-4 \in \delta\cdot(K^{\times})^2 \right\} \bigcup \left\{I,-I\right\}$.

Let us now study the conjugation in $GL_2(K)$ and in $SL_2(K)$. For the demonstration, we note : $S=SL_2(K)$, $G=GL_2(K)$ and $Ext(S)=\{f\in Aut(S) \mid f(M)=M^P\mbox{ for } M\in S, P\in G\}$, $Int(S)=\{f\in Aut(S) \mid f(M)=M^P \mbox{ for }M\in S, P\in S\}$.
Let $P,P'\in G$ and $M\in S$ then : 
$$M^P=M^{P'} \Leftrightarrow P^{-1}MP=P'^{-1}MP' \Leftrightarrow P'P^{-1}M=MP'P^{-1} \Leftrightarrow PP'\in C_G(M)$$
So $P$ and $P'$ define the same automorphism if and only if  $P'P^{-1}\in C_G(S)=Z(G)=K\cdot I_2$, then $Ext(S)\cong GL_2(K)/Z(G) \cong PGL_2(K)$, and similarly $Int(S)\cong SL_2(K)/Z(S) \cong PSL_2(K)$. 
It is known that $PGL_2(K)/PSL_2(K) \cong K^{\times}/(K^{\times})^2$.
Finally $Int(S)$ is a normal subgroup of finite index in $Ext(S)$, and there exist $\mu_1, ..., \mu_n\in GL_2(K)$ such that :
$$Q_{\delta}^{GL_2(K)}= Q_{\delta}^{\mu_1\cdot SL_2(K)} \cup ... \cup Q_{\delta}^{\mu_n\cdot SL_2(K)}$$
\end{proof}

\begin{thm}\label{Cartan}
The subgroups $Q_1$, $Q_{\delta}$ (for $\delta \in K^{\times}\backslash (K^{\times})^2$) and the externally conjugate $Q_{\delta}^{\mu_i}$ (for $\mu_1, ..., \mu_n\in GL_2(K)$) are the only Cartan subgroups of $SL_2(K)$ up to conjugacy.	
\end{thm}

\begin{proof}
It is clear that the image of a Cartan subgroup by an automorphism is also a Cartan subgroup. 
For the demonstration we note $S=SL_2(K)$ and $B$ the following subgroup of $SL_2(K)$ :
$$B=\left\{ \left( \begin{array}{cc}t & u \\0 & t^{-1}\end{array}\right)  \mid t\in K^{\times}, u\in K\right\}$$
With these notations, we can easily check for $g\in U\backslash\{I,-I\}$ that $C_S(g)=U$ and $N_S(U)=B$. Moreover it is clear that every $q\in B$ can be written as $q=tu$ where $t\in Q_1$ and $u\in U$.

Consider $C$ a Cartan subgroup of $SL_2(K)$. We will show that $C$ is a conjugate of $Q_1$ or of one of the $Q_{\delta}^{\mu}$ (for $\delta \in K^{\times}\backslash(K^{\times})^2$ and $\mu\in GL_2(K)$). First we prove $C$ cannot contain a unipotent element other than $I$ or $-I$. Since a conjugate of a Cartan subgroup is still a Cartan subgroup, it suffices to show that $C\cap U=\{I,-I\}$.

In order to find a contradiction, let $u\in C$ be a element of $U$ different from $I$ or $-I$, $u$ is in $C\cap B$. If $\alpha \in N_S(C\cap B)$, then we have that $u^{\alpha}\in C\cap B$, and since $tr(u^{\alpha})=tr(u)=±2$, $u^{\alpha}$ is still in $U$. Therefore $U=C_S(u)=C_S(u^{\alpha})=C_S(u)^{\alpha}=U^{\alpha}$ and so $\alpha$ is in $N_S(U)=B$. 
It follows $N_S(C\cap B) \leq B$ and finally $N_C(C\cap B)=C\cap B$.
By the normalizer condition $C\cap B$ cannot be proper in $C$, then $C\leq B$.

It is known (see for example \cite[Chapter 4, Theorem 2.9]{Suzuki}) that if $C$ is a nilpotent group and $H\unlhd C$ a non trivial normal subgroup, then $H\cap Z(C)$ is not trivial. If we assume that $C\nleq U^+$, since $C\leq B=N_S(U^+)$, $C\cap U^+$ is normal in $C$, and so $C\cap U^+$ contains a non trivial element $x$ of the center $Z(C)$. For $q\in C\backslash U^+$, there are $t\in Q_1\backslash \{I\}$ and $u\in U$ such that $q=tu$. We have $[x,q]=I$ so $[x,t]=I$, so $t=-I$ because $C_S(x)=U$. Therefore $C\leq U$.
Since $C$ is maximal nilpotent and $U$ abelian, $C=U$. But $U$ is not a Cartan subgroup, because it is of infinite index in its normalizer $B$. A contradiction.

Since $C$ does not contain a unipotent element, $C$ intersects a conjugate of $Q_1$ or of one of the $Q_{\delta}^{\mu}$ (for $\delta \in K^{\times}\backslash (K^{\times})^2$ and $\mu\in GL_2(K)$) by Corollary \ref{union}, we note $Q$ this subgroup. Let us show that $C=Q$.
Let be $x$ in $C\cap Q$, and $\alpha\in N_C(C\cap Q)$, then $x^{\alpha}\in Q$, and, by lemma \ref{centre}, $Q=C_S(x^{\alpha})=C_S(x)^{\alpha}=Q^{\alpha}$. Thus $\alpha \in N_S(Q)$, and $N_C(C\cap Q)\leq N_S(Q)$.

\begin{description}
 	 \item[1rst case] $Q$ is a conjugate of $Q_1$, then $N_S(Q)=Q\cdot<\omega'>$ where $\omega'=\omega^g$ if $Q=Q_1^g$. We have also $\omega'^2\in Q$ and $t^{\omega'}=t^{-1}$ for $t\in Q$. %First we can easily see that the cartan subgroup $K$ must be infinite. Otherwise the subgroup of $K$ generated by one element of $Q$ will be of infinite index in his normalizer $Q$ in $SL_2(\Qp)$.
One can check that $N_S(Q\cdot<\omega'>)=Q\cdot<\omega'>$, if $\omega' \in C$ then $N_C(Q\cdot<\omega'>\cap C)=Q\cdot<\omega'>\cap C$, by normalizer condition $C\leq Q\cdot<\omega'>$. If we note $n$ the nilpotency class of $C$, and $t\in C\cap Q$ then $[t, \omega', \omega', ... , \omega']=t^{2^n}=1$, so $t $ is an ${2^n}^{th}$ root of unity, so $C\cap Q$ and $C=(C\cap Q)\cdot <\omega>$ are finite, a contradiction. So $\omega'\nin C$. Then $N_C(Q\cap C)\leq Q\cap C$, it follows by normalizer condition that $C\leq Q$, and by maximality of $C$, $C=Q$.
	  \item[2nd case] $Q$ is a conjuguate of $Q_{\delta}$ (for $\delta \in K^{\times}\backslash (K^{\times})^2$), then $N_S(Q)=Q$. It follows similarly that $C=Q$.
  \end{description}
\end{proof}

 %%%
%\subsection{Description of "maximal" subgroups of $SL_2(K)$}
\begin{cor}\label{nilpotent}
 Let $H$ be a infinite nilpotent subgroup of $SL_2(K)$, then $H$ is a subgroup of a conjugate of either $U$ or $Q_1$ or $Q_{\delta}$ (for some $\delta$ in $\slfrac{K^{\times}}{(K^{\times})^2}$).
\end{cor}

\begin{proof}
The previous demonstration shows that if $H$ is nilpotent and intersects up to conjugacy $U$, $Q_1$ or $Q_{\delta}$, then $H\leq U$, $H\leq Q_1$ or $H\leq Q_{\delta}$ respectively. 
We conclude by the Corollary \ref{union}.
\end{proof}

\begin{rmk}
In particular, every nilpotent subgroup of $SL_2(K)$ is commutatif.
\end{rmk}

For $X$ a set of $K^n$, we denote $\overline{X}^{K}$ the Zariski closure of $X$ in $K^n$, (that means the intersection of all algebraic sets in $K^n$ containing $X$), and $\overline{X}^{\widetilde{K}^{alg}}$ the Zariski closure of $X$ in $\widetilde{K}^{alg} \ ^n$.
We know by \cite[Chap. II, Sect. 5, Théorème 3]{chevalley} that if $H$ is a subgroup of $GL_n(K)$ then $\overline{H}^K$ and $\overline{H}^{\widetilde{K}^{alg}}$ are algebraic groups and $\overline{H}^K=\overline{H}^{\widetilde{K}^{alg}}\cap GL_n(K)$. 
For $Y$ an algebraic group over a field $K$, we denote $Y^{\circ}$ the algebraic connected component of $Y$, this is the intersection of all algebraic subgroups of finite index of $Y$. By \cite[1.2]{borel} it is the smallest algebraic subgroup of finite index of $Y$.

\begin{prop}\label{solvable}
Let $H$ be a a maximal solvable subgroup of $SL_2(K)$.

Then $H$ is the normalizer of a Cartan subgroup or $H$ is a conjugate of the group $B$.
\end{prop}

\begin{proof}
Let us note $\overline{H}^{\circ}$ the algebraic connected component of the  Zariski closure of $H$ in $SL_2(\widetilde{K}^{alg})$. By \cite[2.4]{borel} $\overline{H}^{\circ}$ is a solvable connected subgroup of $SL_2(\widetilde{K}^{alg})$, so $\overline{H}^{\circ}$ is a conjugate of a subgroup of $\overline{B}$ \cite[Theorem 11.1]{borel}, with 
$$B'=\left\{\left(\begin{array}{cc}a & b \\0 & a^{-1}\end{array}\right) \in SL_2(\widetilde{K}^{alg})\mid a\in \widetilde{K}^{alg \times} \mbox{ and } b\in \widetilde{K}^{alg} \right\}$$

If $H\leq \overline{H}^{\circ}$, by Corollary \ref{union} there are 2 possibilities: 
either $\overline{H}^{\circ}$ contains only semi-simple elements or it contains non trivial unipotents elements. In the former case $\overline{H}^{\circ}$ is a torus, so $H$ is also a torus in $SL_2(K)$, i.e. $H$ is a Cartan subgroup (by maximality).
In the latter case, we note $(\overline{H}^{\circ})_u$ the set of all unipotents elements of $\overline{H}^{\circ}$. By \cite[Theorem 10.6]{borel}, $(\overline{H}^{\circ})_u$ is a subgroup and $\overline{H}^{\circ}$ normalizes $(\overline{H}^{\circ})_u$. We can observe that $H_u=(\overline{H}^{\circ})_u\cap H$, so $H$ normalizes $H_u$ and $H$ is a conjugate of $B$.

If not, $H\cap \overline{H}^{\circ}$ is normal subgroup of finite index in $H$. We know that $\overline{H}^{\circ}$ is conjugate to a subgroup of $\overline{B}$, i.e. $\overline{H}^{\circ}\leq \overline{B}^g$ where $g\in SL_2(\widetilde{K}^{alg})$. 
We may assume that $g=I$ and we suppose that $H\nleq \overline{B}$, then there is $h\in H$ such that $\overline{B}^{h}\neq \overline{B}$. 
Therefore $H\cap \overline{H}^{\circ}\leq \overline{B}^{h}\cap \overline{B}$ but $\overline{B}^{h}\cap \overline{B}$ is a torus. It follows that $H$ normalizes the torus $H\cap \overline{H}^{\circ}$, i.e. by maximality, $H$ is the normalizer of a Cartan subgroup. 
\end{proof}

\begin{fact}\label{subgroupsB}
From \cite[Theorem 10.6]{borel}, one can deduce that definable subgroups of $B$ are as follow : 
$$\left\{\left(\begin{array}{cc}a & b \\0 & a^{-1}\end{array}\right)\in SL_2(K) \mid a\in P \mbox{ and } b\in Z\right\}$$
where $P$ is a definable subgroup of $K^{\times}$ and $Z$ is a definable subgroup of $K^+$ such that $P\cdot Z\subseteq Z$.
\end{fact}

% Definissabilité
%\subsection{Remark on definissability}

Corollary \ref{nilpotent} and Proposition \ref{solvable} are true for any subgroup of $SL_2(K)$, independent from any  assumption on definability or algebraicity.
$Q_1$, $Q_{\delta}$, $U$ and $B$ appear like "frame" subgroups which contain, up to conjugacy, every nilpotent or solvable subgroups of $SL_2(K)$.
These "frame" subgroups are definable in the pure language of groups. This means that we consider $SL_2(K)$ as a first order structure and these subgroups are definable in the language $\mathcal{L}_G=\{\cdot, \ ^{-1}, e\}$ :
every nilpotent subgroup $Q_1$, $Q_{\delta}$ or $U$ is the centralizer of one of its non central element and $B=N_S(C_S(u))$ where $u\in U\backslash\{I,-I\}$.
Obviously, they are also definable in the structure $K$ with the ring language $\mathcal{L}_{R}=\{+,-,\cdot ,0,1\}$. These descriptions are independent from the choice of field. Nevertheless the structures of these groups are very different depending on field. In order to understand the finer structure of these groups, we need to further investigate the model-theoretical nature of the field and answer the question : What are the subgroups of $Q_1$, $Q_{\delta}$, $U$ or $B$, definable in the field $K$ ? 

We finish this section with related remark that every group definable in the field language $\mathcal{L}_R$ is interpretable in the pure group and moreover to be interpretable in the field language is equivalent to be interpretable in the pure group:

\begin{prop}
Let $K$ be an infinite field of characteristic different from $2$. Then the field $(K,+,-,\cdot,0,1)$ is interpretable in the pure group $(SL_2(K),\cdot)$.
\end{prop}

\begin{proof}The subgroup $Q_1$ acts on $U$ :
$$\left(\begin{array}{cc}t & 0 \\0 & t^{-1}\end{array}\right)\left(\begin{array}{cc}1 & u \\0 & 1\end{array}\right)\left(\begin{array}{cc}t^{-1} & 0 \\0 & t\end{array}\right)=\left(\begin{array}{cc}1 & t^2 u \\0 & 1\end{array}\right)$$

For the demonstration, we identify the matrix $\left(\begin{array}{cc}t & 0 \\0 & t^{-1}\end{array}\right)\in Q_1$ and $t\in K^{\times}$ and we choose one element $u_0$ in $U$ and its inverse $u_1$.

We consider $A$ the set $Q_1\times Q_1$ quotiented by the equivalence relation : 
$$(t_0,t_1)\sim (t_0',t_1') \quad \mbox{ iff } \quad u_0^{t_0}\cdot u_1^{t_1}=u_0^{t_0'}\cdot u_1^{t_1'} \mbox{ in }SL_2(K) \quad (\mbox{ i.e. } \quad t_0^2-t_1^2=t_0'^2-t_1'^2 \mbox{ in } K)$$

We have to construction a bijection between $A$ and $K$ and to define the addition and the multiplication of the field in $A$.
First we remark that for every $x\in K$ $$x=\left(\frac{x+1}{2}\right)^2-\left(\frac{x-1}{2}\right)^2$$
Then, we see that the application $\varphi : A\longrightarrow K$, $(t_0,t_1)\mapsto t_0^2-t_1^2$ is a bijection. 

Now the addition and the mulplication of $K$ are given by : 
$$ (t_0,t_1)+(t_0',t_1')=(t_0'',t_1'') \quad \mbox{ iff } \quad u_0^{t_0}u_1^{t_1}{u_0}^{t_0'}u_1^{t_1'}=u_0^{t_0''}u_1^{t_1''}$$
$$ (t_0,t_1)\cdot(t_0',t_1')=(t_0'',t_1'') \quad \mbox{ iff } \quad u_0^{t_0t_0'}u_0^{t_1t_1'}{u_1}^{t_0t_1'}u_1^{t_1t_0'}=u_0^{t_0''}u_1^{t_1''}$$
\end{proof}

For $K$ an algebraically closed field, $Q_1\cong K^{\times}$ and $U\cong K^+$, and they have no proper infinite definable subgroups. If $K$ is a real closed field, $Q_1\cong K^{\times}$ so $\{1,-1\}$ and $K^{>0}$  are its only definable subgroups, and $U\cong K^{+}$ has no proper definable subgroup. Up to conjugacy, there is one Cartan subgroup which  contains non central diagonalizable element : $Q_{-1}=SO_2(K)$, and one can easily check that it has no proper infinite definable subgroup. The question for the $p$-adically closed field case is less trivial. 
We will analyze this issue in the next section by identifying the subgroups of $SL_2(K)$, with $K$ $p$-adically closed, definable in the field language.

% CAS P-ADIQUE
\section{The $p$-adically closed case}\label{padique}
During this section, unless otherwise noted,  $K$ will be a $p$-adically closed field. By $p$-adically closed we mean elemenetary equivalent to $\Qp$ $K$ (for some authors this is equivalent to be $p$-adically closed of $p$-rank 1). This means that $(K,v_p)$ is a henselian valued field of characteristic $0$ whose value group $\Gamma$ is a $\mathbb{Z}$-group (i.e. elementarily equivalent to the ordered additive group $(\mathbb{Z},+, <)$ ),  residue field $k$ is $\mathbb{F}_p$, and such that $v_p(p)$ is the smallest positive element of $\Gamma$. 

On the one hand, we can study $\Qp$ in the usual field language $\mathcal{L}_R=\{+,-,\cdot,0,1\}$. By quantifier elimination in the language $\mathcal{L}_R\cup \{P_n(x)\}_{n\geq 1}$ where $P_n(x)$ are predicates for $\exists y\ y\neq 0 \wedge x=y^n$ \cite{MacEQ}, it is kown that definable sets in $K$ are precisely semi-algebraic sets, which are finite boolean combinations of sets defined by $f(\bar{x})=0$ or $P_n(g(\bar{x}))$ with $f(\bar{X}), g(\bar{X})\in K[\bar{X}]$.
In the other hand, we can  consider $\Qp$ as a 3-sorted structure $(K, k, \Gamma)$ where $K$ is the field, $k$ is the residual field and $\Gamma$ the values group. For each field $K$ and $k$, we have symbols from the field language and $\Gamma$ is considered as an ordered group in the language $\mathcal{L}_{OG}=\{+,-,<,0,1\}$, we add two symbols for the valuation $v:K\rightarrow \Gamma\cup\{\infty\}$ and the residue map $res: K\rightarrow k$. 
Since the valuation ring $\mathcal{O}$ is definable in the $\mathcal{L}_R$, the two previous points of vue have the same expressive power (what is definable in the one is definable in the other). 

From the model-theorical viewpoint, $\Gamma$ is a $\mathbb{Z}$-group, it verifies the Presburger arithmetic \cite[p. 81]{Marker}. We know that $\Gamma$ admits quantifier elimination in the language $\mathcal{L}_{OG}\cup \{S_n(x)\}_{n\in \mathbb{N}}$ where $S_n(x)$ are predicates for $\exists y \ x=ny$\cite[Corollary 3.1.21]{Marker}, then every definable subset in $\Gamma$ is a finite union of intervals and of sets of the form $n\Gamma$. So we can conclude that definable subgroups of $\Gamma$ are of the form $n\Gamma$ for $n\in \mathbb{N}$ (an interval is not stable under addition).

We consider $\Qp^{an}$ the expansion of $\Qp$ by adjoining "all restricted analytic functions", that is functions $f:\Zp^m\longrightarrow \Qp$ given by power series $\sum_I a_I x^I$ in $\Zp[[\bar{x}]]$ such that $v(a_I)$ tends to infinity as $\mid I\mid \rightarrow \infty$. $\Qp^{an}$ is studied in the language $\mathcal{L}_{an}$ obtained from $\mathcal{L}_R$ by adjoining symbols for these new functions. 

Most statements in this section are true for more general fields than the $p$-adically closed one. We will try to formulate each statement in the largest possible context known to us. Indeed, except the propositions from part 2.2, everything works for finite extension of $p$-adically closed fields. Moreover the raisoning using dimension end topological properties assure us that statement stay true for $\Qp^{an}$. 

For this study, we use model theory of $p$-adically closed fields, especially a notion of dimension.

% dimension
\subsection{Definable subgroups and dimension}
We will use the notion  of dimension introduced  by van den Dries in \cite{dimvdD}. Axioms used in this definition seem to be the most general, they  imply in particular the notion used in \cite{cartan}. 
We wil recall principal facts about it and refer to the article \cite{dimvdD} for technical points. The aim of this section is to establish  a direct link between dimension and algebraic properties of definable subgroups (Proposition \ref{dimension}).

We work with a structure $M$ such that  each nonempty set definable in $M$ is equipped with a dimension in $\mathbb{N}\cup\{ -\infty\}$ satisfying the following axioms:

For any definable sets $S$, $S_1$ and $S_2$ :
\begin{enumerate}[label=(Dim \arabic*)]
	\item $\dim S=-\infty \Leftrightarrow S=\emptyset$,
	$\dim \{ a\}=0$ for each $a\in M$,
	$\dim M^1=1$.
	\item $\dim( S_1 \cup S_2)=\max (\dim S_1, \dim S_2)$.
	\item $\dim S^{\sigma}=\dim S$ for each permutation $\sigma$ of $\{1, ..., m\}$, where 
	$$S^{\sigma}=\{(x_{\sigma(1)}, ... , x_{\sigma(m)})\in M^m \mid (x_1, ... , x_m)\in S\}$$
	\item If $T\subset M^{m+1}$ is a definable set and $T_x=\{y\in M \mid (x,y)\in T\}$ for each $x\in M^m$, then $T(i)=\{x\in M^m \mid \dim T_x=i\}$ (for $i=0,1$) is definable and $$\dim \{(x,y)\in T \mid x\in T(i)\}=\dim T(i) +i$$
\end{enumerate}

These very general axioms imply \cite[1.1 and 1.5]{dimvdD} more precise and pratically useful properties: 
 
\begin{description}
	\item[Definability] If $f$ is definable function from $S_1$ to $S_2$ then the set $\{y\in S_2\mid \dim(f^{-1}(y))=m\}$ is definable  for every $m$ in $\mathbb{N}$.
	\item[Additivity] If $f$ is a definable function from $S_1$ to $S_2$, whose fibers have constant dimension $m$ in $\mathbb{N}$, then $\dim S_1=\dim Im(f) + m$. In particular $\dim (S_1\times S_2)=\dim S_1 +\dim S_2$
	\item[Finite sets] $S$ is finite iff $\dim S =0$.
	\item[Monotonicity] If $f$ is a definable function from $S$ to $A^m$ then $\dim f(S) \leq \dim S$, and if $f$ is injective $\dim f(S)=\dim S$.  In particular, if $S_1\subseteq S_2$ then $\dim S_1 \leq \dim S_2$.
\end{description}

Van den Dries showed \cite{dimvdD} that henselian fields of characteristic $0$ are equipped with such a notion of dimension. In the case of $\Qp$, it corresponds to the notion defined in \cite{Dries}.
If $(K,v)$ is a valued field, the valuation $v$ define a topology and we will note $\overline{X}^v$ the closure of the set $X\subseteq K^n$ for this topology.

\begin{lem}\label{lemdim}
Let $K$ be a henselian field of characteristic 0, and $X$ and $Y$  sets in $K^m$ definable in the field language. Then
\begin{itemize}
	\item	the dimension is compatible with the algebraic closure, i.e. : 
	$$\dim_K X = \dim_K \overline{X}^K=\dim_{\widetilde{K}^{alg}} \overline{X}^{\widetilde{K}^{alg}}$$
	\item the dimension is compatible with the topology, i.e.:
	$$ \mbox{ if } X\subseteq Y \mbox{ and } \dim X=\dim Y \quad \mbox{ then } \quad X \mbox{ has non empty interior in } Y$$
\end{itemize}

\end{lem}

In the first point the dimension $\dim _{\widetilde{K}^{alg}}\overline{X}^{\widetilde{K}^{alg}}$ is unterstood as the algebro-geometric dimension of the Zariski closure in $\widetilde{K}^{alg}$.

\begin{proof}
$\bullet $ We know by \cite[2.12]{dimvdD} that  $\dim_K X=\dim _K \overline{X}^K$.
Let $K'\succeq K$ a $|K|^+$-saturated structure and $X^*$ is the set of $K'^m$ defined by the same formula as $X$. We have by \cite[1.7]{dimvdD} and \cite[2.3]{dimvdD} : 
$$\dim_K X=\max \{ trdeg _K K(x) \mid x\in X^*\}=\dim_{\widetilde{K'}^{alg}}X^* =\dim _{\widetilde{K'}^{alg}} \overline{X^*}^{\widetilde{K'}^{alg}}=\dim _{\widetilde{K}^{alg}} \overline{X}^{\widetilde{K}^{alg}}$$
By \cite[2.1]{dimvdD}, we do not need $X^*$ to be definable in $\widetilde{K'}^{alg}$.

$\bullet$ Let $X\subseteq Y$ be definable sets such that $\dim X=\dim Y$ and in order to find a contradiction suppose that $X$ has empty interior in $Y$. That means that $X$ does not contain any open of $Y$, so for all $x\in X$ and $\gamma \in \Gamma$, $B_{\gamma}(x)\cap Y \nsubseteq X$ $i.e.$ there exists $y\in Y\backslash X$ such that $y\in B_{\gamma}(x)\cap Y$. This means that $Y\backslash X$ is dense in $Y$ for the valuation topology, so $\overline{Y\backslash X}^v=Y$. Using \cite[2.23]{dimvdD}, we find $\dim ((\overline{Y\backslash X}^v)\backslash(Y\backslash X))< \dim (Y\backslash X)$ and finally $\dim X<\dim Y$, a contradiction. 
\end{proof}

\begin{rmk}
We can easily check that we have  $\dim Q_1 =\dim Q_{\delta} = \dim U = 1$, and $\dim B=2$. Thus Corollary \ref{nilpotent} and Proposition \ref{solvable} show that : 
\begin{enumerate}
	\item If $H$ is a definable nilpotent subgroup of $SL_2(K)$ then $\dim H=1$.
	\item If $H$ is a definable solvable subgroup of $SL_2(K)$ then $\dim H\leq2$.
\end{enumerate}
The next proposition give us the converse. 
\end{rmk}

\begin{prop} \label{dimension}[$K$ $p$-adically closed field]
Let $H$ be an infinite definable subgroup of $SL_2(K)$.
\begin{enumerate}
	\item $\dim H=1$ iff $H$ is commutative or $H$ is a subgroup of a conjugate of $N_{SL_2(K)}(Q_1)$. 
	\item $\dim H= 2$ iff $H$ is a non nilpotent subgroup of a conjugate of $B$.
	\item $\dim H=3$ iff $H$ is not solvable.
\end{enumerate}
\end{prop}

\begin{proof}
By proposition \ref{solvable}, it suffices to show the first two points : 
\begin{enumerate}
	\item $\overline{H}^{\circ}$ is of dimension 1 in $\widetilde{K}^{alg}$, so by \cite[20.1]{humphreys}, $\overline{H}^{\circ}$ is comutative. 
	\begin{itemize}
		\item if $H<\overline{H}^{\circ}$, $H$ is commutative ; 
		\item if not, $H\cap \overline{H}^{\circ}$is a normal subgroup of finite index in $H$, then $H$ normalize a commutative subgroup of finite index. So $H\leq N_{SL_2(K)}(T)$ where $T$ is conjugate to $Q_1$. 
	\end{itemize}
	\item $\overline{H}^{\circ}$ is of dimension 2 in $\widetilde{K}^{alg}$ (in particular it is connected group of Morley rank 2), by \cite[Theorem 6]{cherlin}, it is solvable. We know by the demonstration of proposition \ref{solvable}, that $H$ is a conjugate of a subgroup of $B$. 
\end{enumerate}
\end{proof}

\begin{rmk}\label{Qpan}
By \cite{vdDHasMac}, we know that $\Qp^{an}$ is equipped with such a notion of dimension. Moreover the dimension verifies the following propertie : $$ \mbox{ if } X\subseteq Y \mbox{ and } \dim X=\dim Y \quad \mbox{ then } \quad X \mbox{ has non empty interior in } Y$$
However dimension in $\Qp^{an}$ is not compatible with algebraic closure. 

Thus for $K$ a $\mathcal{L}_{an}$ elementary extension of $\Qp^{an}$ and $H$  an infinite definable subgroup of $SL_2(K)$, we have the following implications :
\begin{enumerate}
	\item $\dim H=1$ $\Leftarrow$ $H$ is commutative or $H$ is a subgroup of a conjugate of $N_{SL_2(K)}(Q_1)$. 
	\item $\dim H= 2$ $\Leftarrow$ $H$ is a non nilpotent subgroup of a conjugate of $B$.
	\item $\dim H=3$ $\Rightarrow$ $H$ is not solvable.
\end{enumerate}

\end{rmk}

% Commutative subgroups
\subsection{Commutative definable subgroups}

In this section we are interested in the description of definable commutative subgroups of $SL_2(K)$ where $K$ is a $p$-adically closed field. Definable means here definable in $\mathcal{L}_R$ or in $\mathcal{L}_{an}$. We already know that they are, up to conjugacy, subgroups of $U$, $Q_1$ or $Q_{\delta}$ (with $\delta\in K^{\times}\backslash (K^{\times})^2$).Thus, we have to describe the definable subgroups of $U$, $Q_1$ and $Q_{\delta}$. 

We know that $U_+\cong K^+$, thus definable subgroups of $U_+$ correspond to definable subgroups of $K^+$. A subgroup of $K^+$ is infinite so of dimension 1, thus it is an open subgroup of $K^+$.
Open subgroups of $\Qp^+$ are of the form $p^n\Zp$ for $n\in \mathbb{Z}$ \cite[Lemma 3.2]{pillayfields}, so they are the only definable subgroups of $\Qp^+$ (and of $(\Qp^{an})^+$). For $K$ a $p$-adically closed field, we have the same property : definable subgroups of $K^+$ are of the form $a_{\gamma}\mathcal{O}$ where $v_p(a_{\gamma})=\gamma\in \Gamma$.

For $Q_1\cong K^{\times}$, let us show the following result  : 

\begin{prop}
Let $K$ be a $p$-adically closed field and $H$ an infinite definable subgroup of $K^{\times}$.
\begin{enumerate}
	\item If $H$ is bounded, then there exists $\gamma_0\in \Gamma$ and $a_{\gamma_0}\in K$ with $v_p(a_{\gamma_0})=\gamma_0$ such that $H$ contains $1+a_{\gamma_0}\mathcal{O}$ as a subgroup of finite index at most $(p-1)$ for $p\neq 2$ and at most $2$, for $p=2$. 
	\item If $H$ is unbounded then there exist $\gamma_0\in \Gamma$, $n\in\mathbb{N}$ and $\{a_{\gamma}\}_{\gamma\in\Gamma} \subseteq K$ and $b_{\gamma_0}\in K$ with $v_p(a_{\gamma})=\gamma$ and $v_p(b_{\gamma_0})=\gamma_0$, such that $H$ contains $\{a_{\gamma} ; \gamma\in n\Gamma\}\cdot (1+b_{\gamma_0}\mathcal{O})$ as subgroup of finite index at most $(p-1)$  for $p\neq 2$ and at most $2$, for $p=2$. 
\end{enumerate}
\end{prop}

\begin{proof}
\begin{enumerate}
	\item We assume $p\neq 2$ and we first work  in $\Qp$. Since $H$ is bounded, $H\leq \Zp^{\times}$. Let us denote $H_0$ the torsion-free part of $H$, then $H_0\leq 1+p\Zp$. It is well-known that $1+p\Zp\cong (\Zp,+)$ and let us  following the reasoning of Pillay in $\Zp$ \cite[Lemma 3.2]{pillayfields}. Since $\dim H_0=1=\dim (1+p\Zp)$,  $H_0$ contains a open neighborhood of $1$. Thus there exists $n\in \mathbb{N}$ such that $1+p^n\Zp\subseteq H_0$. Let $n_0$ be the smallest such $n$. Let us show that $H_0=1+p^{n_0}\Zp$. In order to find a contradiction, let $x$ be in $H_0$ such that $x\nin 1+p^{n_0}\Zp$. It is  easy to remark that if $x\in 1+p^n\Zp$, then $x^p\in 1+p^{n+1}\Zp$, so by replacing $x$ by a suitable $p^{th}$ power of $x$, we can suppose that $x\in 1+p^{n_0-1}\Zp$. 
	As $\slfrac{(1+p^{n_0-1}\Zp)}{(1+p^{n_0}\Zp)}\cong \slfrac{\mathbb{Z}}{p\mathbb{Z}}$ and $\{x^i ; 0\leq i<p\}$ forms a complete set of cosets representatives of $1+p^{n_0}\Zp$ in $1+p^{n_0-1}\Zp$, so $1+p^{n_0-1}\Zp\subseteq H_0$, a contradiction and $H_0=1+p^{n_0}\Zp$. 
	The number of torsion  elements of $H$ is finite and at most $(p-1)$, so $H_0$ is of finite index at most $(p-1)$ in $H$. Then, as $1+p^n\Zp$ is definable, we shown that 
	\begin{align} \Qp\models \forall \bar{a} \ ("\varphi(x,\bar{a}) \mbox{ defines a subgroup of } \Qp^{\times}")  \rightarrow &  \exists b \ \exists x_1,..., x_{p-1} \bigwedge_{0\leq i\leq p-1} \varphi(x_i,\bar{a})  \nonumber\\
	& \wedge\forall y \  (\varphi(y,\bar{a}) \rightarrow \bigvee_{0\leq i\leq p-1} y\cdot x_i^{-1}\in 1+b\Zp)\nonumber
	\end{align}
	Then the property is true for every $p$-dically closed field, and it finish the proof.
	For $p=2$, the same proof works remplacing $1+p\Zp$ by $1+4\mathbb{Z}_2$. 
	\item We call $H_1=H\cap \mathcal{O}^{\times}$. We can easily check that two elements of $H$ are in the same coset of $H_1$ if and only if they have the same valuation. Moreover, as $\Gamma$ is a $\mathbb{Z}$-group, and $v_p(H)$ is a definable subgroup of $\Gamma$ then $v_p(H)$ is of the form $n\Gamma$ for some $n\in \mathbb{N}$. Then we can choose some $a_{\gamma}\in H$ such that $v_p(a_{\gamma})=\gamma$ and $\{a_{\gamma} ; \gamma \in n\Gamma\}$ forms a set of coset representatives of $H_1$ in $H$. We know by 1. that there exist $\gamma_0\in \Gamma$ and $b_{\gamma_0}\in K$ with $v_p(b_{\gamma_0})=\gamma_0$ such that $1+b_{\gamma_0}\mathcal{O}$ is of finite index at most $p-1$  in $H_1$ if $p\neq 2$ (and at most $2$ if $p=2$), so do $\{a_{\gamma};\gamma\in n \Gamma\}\cdot (1+b_{\gamma_0}\mathcal{O})$ in $H$.
\end{enumerate}
\end{proof}

%\begin{fact}
%If $K$ is a $p$-adically closed field and $H$ a infinite definable subgroup of $K^{\times}$. Then there exists $n\in \mathbb{N}$ such that $(K^{\times})^n$ or $(\mathcal{O}^{\times})^n$ is a subgroup of finite index in $H$. 
%\end{fact}

%\begin{proof}
%Let us reason in $\Qp$. If $H\leq \Zp^{\times}$, then it is compact and of finite index $n$ in $\Zp^{\times}$ hence $(\Zp^{\times})^n\leq H$, but $(\Zp^{\times})^n$ is open, so of finite index in $\Zp^{\times}$ and in $H$. If $H\nleq\Zp^{\times}$, let $m=\{n>0 \mid \exists x\in H\quad v_p(x)=n\}$, then $H\cdot \Zp^{\times}=\{x\in \Qp^{\times}\mid v_p(x)\in m\Z\}$ has finite index in $\Qp^{\times}$. As $H$ is open and $\Zp^{\times}$ is compact, by the second theorem of isomophism $H$ has finite index in $H\cdot\Zp^{\times}$ so in $\Qp^{\times}$, say index $n$. Again $(\Qp^{\times})^n\leq H$ and as $(\Qp^{\times})^n$ is of finite index in $\Qp^{\times}$ thus also in $H$.

%For $K$ a $p$-adically closed field, $H$ an infinite subgroup of $K^{\times}$ defined by $\varphi(x)$. Then $\varphi(\Qp)$ is a infinite definable subgroup of $\Qp^{\times}$, we know that $\varphi(\Qp)$ contains $(\Zp^{\times})^n$ or $(\Qp^{\times})^n$ as subgroup of finite index, what is expressible in a first order formula, and so true for $K$. 
%\end{proof}

The aim is now to study definable subgroups of $Q_{\delta}$. Let us remark first that: 

\begin{fact}[\cite{JPS}]\label{Serre}
Let $K$ be a $p$-adically closed field. 

If $p\neq2$, the group $K^{\times} / (K^{\times})^2$ is isomorphic to $\slfrac{\Z}{2\Z} \times \slfrac{\Z}{2\Z}$, it has for representatives $\{1, \alpha, p, \alpha p\}$, where $\alpha\in \mathcal{O}^{\times}$ is such that $res(\alpha)$ is not a square in $\mathbb{F}_p$

If $p=2$, the group $K^{\times}/(K^{\times})^2$ is isomorphic to $\slfrac{\Z}{2\Z} \times \slfrac{\Z}{2\Z} \times \slfrac{\Z}{2\Z}$, it has for representatives $\{\pm1,\pm 2, \pm 5, \pm10\}$. 
\end{fact}

Then there are three $Q_{\delta}$ up to conjugacy for $p\neq 2$ (and seven for $p=2$). Everything we can do here remains true up to conjugacy  for every non square $\delta$. 

We will first work in $\Qp$ and then generalize to arbritrary model of $Th(\Qp)$. We need to separate the case $p=2$ and $p\neq 2$  : the results are similar except for specific values and the  demonstrations are the same  mutatis mutandis. That  is why, after explaining specificities of the two cases, we will work on case $p\neq 2$. 

We will fixe  special values of $\delta$: $\delta$ will be one of the representative elements $\{ \alpha , p, \alpha p \}$ for the non square in $\slfrac{\Qp^{\times}}{(\Qp^{\times})^2}$ if $p\neq 2$ and one of the $\{-1 ; \pm 2 ; \pm 5 ; \pm 10\}$ for $p=2$. In any case, we have $0\leq v_p(\delta)\leq 1$. With these notations, we can remark that $Q_{\delta}\subseteq SL_2(\Zp)$.
We put : 
$$\mbox{for }p\neq 2, n\geq0 \quad Z_{n,\delta}:=\left\{\begin{pmatrix}
a & b \\ b\delta & a
\end{pmatrix}\in SL_2(\Qp) \mid b\in p^n\Zp, a\in 1+p^{2n}\delta\Zp \mbox{ and } a^2-b^2\delta=1\right\}$$
$$\mbox{for }p= 2, n\geq 1 \quad Z_{n,\delta}:=\left\{\begin{pmatrix}
a & b \\ b\delta & a
\end{pmatrix}\in SL_2(\mathbb{Q}_2) \mid b\in p^n\mathbb{Z}_2, a\in 1+p^{2n-1}\delta\mathbb{Z}_2 \mbox{ and } a^2-b^2\delta=1\right\}$$
In order to simplify the notations, we will denote by $(a,b)$ an element of $Q_{\delta}$.

\begin{lem}
  
\begin{itemize}
	\item if $p\neq 2$ and for $(a,b)\in Z_{0, \delta}$, then \quad $b\in p^n\Zp^{\times}$ iff $a\in 1+p^{2n}\delta\Zp^{\times}$.
	\item if $p=2$ and for $(a,b)\in Z_{1,\delta}$, then  \quad $b\in p^n\mathbb{Z}_2^{\times}$ iff $a\in 1+p^{2n-1}\delta\mathbb{Z}_2^{\times}$.
\end{itemize}
\end{lem}

\begin{proof}
$\bullet$ We first deal with the case $p\neq 2$, let be $(a,b)\in Z_0$ with $b\in p^n\Zp^{\times}$ then $a$ is such that $a^2-b^2\delta =1$.
We have $a^2=1+b^2\delta$ so $a^2\in 1+p^{2n}\delta\Zp^{\times}$. We first show that $a^2\in 1+p^{2n}\delta\Zp^{\times}$ if and only if $a\in 1+p^{2n}\delta\Zp^{\times}$: 
\\ if $a=1+p^{2n}\delta u$ with $u\in \Zp^{\times}$ then $a^2=1+2p^{2n}\delta u+p^{4n}\delta^2u^2\in 1+p^{2n}\Zp^{\times }$ (because $p\neq2$) ;
\\ if $a=1+p^{k}\delta u$ with $k\neq 2n$, then $a^2\in 1+p^k\delta\Zp^{\times}$ and $a^2\nin 1+p^{2n}\delta\Zp^{\times}$.

Now if $a\in 1+p^{2n}\delta \Zp^{\times}$, we show that $v_p(b)=n$. We have $b^2\delta=a^2-1\in p^{2n}\delta \Zp^{\times}$, so $2v_p(b)+v_p(\delta)=2n+v_p(\delta)$ then $v_p(b)=n$.

$\bullet$ For $p=2$ we argue in the same way. We just need to prove that if $a\in 1+2^k\delta\mathbb{Z}_2^{\times}$, then $a^2\in 1+2^{k+1}\delta\mathbb{Z}_2^{\times}$: if $u\in \mathbb{Z}_2^{\times}$, $(1+2^{k}\delta u)^2=1+2^{k+1}\delta(u+2^{k-1}\delta u^2)\in 1+2^{k+1}\delta \mathbb{Z}_2^{\times}$.
\end{proof}

From now, we assume $p\neq 2$, the same proof will work for $p=2$ mutatis mutandis. 

We can remark that $Z_{n,\delta}$ is a subgroup : 
Let be $x,y\in Z_{n,\delta}$, $x=(1+p^{2n}\delta a, p^nb)$ and $y=(1+p^{2n}\delta a', p^nb')$.
$$xy=(1+p^{2n}\delta(a+a'+bb'+p^{2n}\delta aa'), p^n(b+b'+p^{2n} \delta ab'+p^{2n}\delta a'b))\quad (\star)$$
$$x^{-1}=(1+p^{2n}\delta , -p^n b)$$

We see then groups $Z_{n,\delta}$ form an infinite descending chain of definable subgroups. 
Before to show the main proposition, let us etablish some technical lemma. 

\begin{lem}\label{L1}
For $p\neq 2$, $\delta\in \{p,\alpha p\}$ and $n\geq 0$ ( or for $\delta=\alpha$ and $n\geq 1$).
\begin{enumerate}
	\item $\slfrac{Z_{n,\delta}}{Z_{n+1, \delta}}\cong \slfrac{\mathbb{Z}}{p\mathbb{Z}}$
	\item If $x \in Z_{n,\delta}\backslash Z_{n+1,\delta} $ then $x^{p^k}\in Z_{n+k,\delta}\backslash Z_{n+k+1,\delta}$
\end{enumerate}
\end{lem}

\begin{proof}
\begin{enumerate}
	\item We define : $$\begin{array}{crcc} \varphi : & Z_{n,\delta} & \longrightarrow & \slfrac{\Zp}{p\Zp}
\\&(1+p^{2n}\delta a,p^n b) &\longmapsto & b \mod p\Zp\end{array}$$
$(\star)$ show that $\varphi$ is a well-defined surjective homomophism from the group $Z_{n,\delta}$ to the additive group $\slfrac{\Zp}{p\Zp}$. Its kernel is $Z_{n+1,\delta}$, so $\slfrac{Z_{n,\delta}}{Z_{n+1,\delta}}\cong \slfrac{\mathbb{Z}}{p\mathbb{Z}}$.
	\item Let us show by induction that : 
	$$\mbox{if } x=(1+p^{2n}\delta a, p^nb)\mbox{, then } x^k=(1+p^{2n}\delta a' ,p^n(kb+p^{2n}\delta b'))$$
	For $k=1$, this is obvious. 
	
	Suppose now that $x^k=(1+p^{2n}\delta a' ,p^n(kb+p^{2n}\delta b'))$, then :
\begin{align}
x^kx & = (1+p^{2n}\delta a',p^n (kb+p^{2n}\delta b'))(1+p^{2n}\delta a,p^n b) \nonumber \\	
	& = (1+ p^{2n}\delta(a+a'+b(kb+p^{2n}\delta b')+p^{2n}\delta aa'),p^n((k+1)b+p^{2n}\delta b'+p^{2n}\delta a'b+p^{2n}\delta(a(kb+p^{2n}\delta)))) \nonumber \\
x^{k+1}&= (1 + p^{2n}\delta a'',p^n((k+1)b+p^{2n}\delta b'')) \nonumber 
\end{align}
	So if $x\in (1+p^{2n}\delta \Zp^{\times}, p^n\Zp^{\times}) \subseteq Z_{n,\delta}\backslash Z_{n+1,\delta}$ then $x^p=(1+p^{2n}\delta A, p^{n+1} B)$ with $B\in \Zp^{\times}$, and by the initial remark $A\in p^2\Zp^{\times}$, i.e. $x^p\in Z_{n+1,\delta}\backslash Z_{n+2,\delta}$. 
	
	Another induction shows that $x^{p^k}\in Z_{n+k,\delta}\backslash Z_{n+k+1,\delta}$.
\end{enumerate}
\end{proof}

\begin{lembis}{L1}
For $p=2$, $v_p(\delta)=1$ and $n\geq 1$ ( or for $v_p(\delta)=0$ and $n\geq 1$).
\begin{enumerate}
	\item $\slfrac{Z_{n,\delta}}{Z_{n+1, \delta}}\cong \slfrac{\mathbb{Z}}{p\mathbb{Z}}$
	\item If $x \in Z_{n,\delta}\backslash Z_{n+1,\delta} $ then $x^{p^k}\in Z_{n+k,\delta}\backslash Z_{n+k+1,\delta}$
\end{enumerate}
\end{lembis}

\begin{prop}\label{L2}
For $p\neq 2$ and $\delta \in\{p,\alpha p\}$ (resp. for $\delta=\alpha$).
\begin{enumerate}
	\item $Z_{0,\delta}$ (resp. $Z_{1,\delta})$ is of finite index in $Q_{\delta}$.
	\item The $Z_{n,\delta}$ are the only one subgroups of $Z_{0,\delta}$ (resp. $Z_{1,\delta}$) definable in $\Qp^{an}$.
\end{enumerate}
\end{prop}

\begin{proof}
We will work with $\delta\in \{p, \alpha p\}$, the other case is similar :
\begin{enumerate}	
	\item We consider $\Qp(\sqrt{\delta})$ the quadratic extension of $\Qp$ and $k'$ its residue field, $k'$ is a finite field. 
	Let $\psi$ be the following group homomorphism  : $$\begin{array}{crcc} \psi : & Q_{\delta} & \longrightarrow & k'^{\times}
\\&(a,b) &\longmapsto & res(a+b\sqrt{\delta}) \end{array}$$
We see that $\ker \psi =Z_{0,\delta}$ and $Z_{0,\delta}$ is of finite index. 	
	\item Let $H\leq Z_{0,\delta}$ be a nontrivial definable subgroup. Since $Z_{0,\delta}$ does not have torsion, $H$ is infinite. Then $\dim H=1=\dim Z_{0,\delta}$ and $H$ has non empty interior in $Z_{0,\delta}$, so $H$ is open in $Z_{0,\delta}$. Now it will suffice to show that the $Z_{n,\delta}$ are the only open subgroups. 
	
	The $Z_{n,\delta}$ form an open neighborhood basis of $I$ in $Q_{\delta}$. Let $H\leq Z_{0,\delta}$ be an open subgroup. $H$ contains some open neighborhood $Z_{n,\delta}$. We note $n_0$ the smallest such that $Z_{n_0,\delta}\subseteq H$. If $H\neq Z_{n_0,\delta}$, then there exists $x\in H\cap Z_{n_1,\delta}$ with $n_1<n_0$. Replacing $x$ with some $x^{p^k}$, we may assume that $x\in Z_{n_0-1,\delta}\backslash Z_{n_0,\delta}$.  We have $\slfrac{Z_{n_0-1,\delta}}{Z_{n_0,\delta}}\cong \slfrac{\mathbb{Z}}{p\mathbb{Z}}$ so $x^t$ with $0\leq t\leq p$ is a complete system of representatives of class modulo $Z_{n_0,\delta}$. So $Z_{n_0-1,\delta}\subseteq H$ and  this contradicts minimality of $n_0$.
\end{enumerate}
\end{proof}

\begin{propbis}{L2}
For $p= 2$ and $v_p(\delta)=1$ (resp. for $v_p(\delta)=0$).
\begin{enumerate}
	\item $Z_{1,\delta}$ (resp. $Z_{2,\delta})$ is of finite index in $Q_{\delta}$.
	\item The $Z_{n,\delta}$ are the only subgroups of $Z_{1,\delta}$ (resp. $Z_{2,\delta}$) definable in $\mathbb{Q}_2^{an}$.
\end{enumerate}
\end{propbis}

\bigskip

We can now generalize to $K$  $p$-adically closed.
We consider a sequence $(a_{\gamma})_{\gamma\in \Gamma}$ of elements of $K$ indexed by the value group such that $v_p(a_{\gamma})=\gamma$. We have 
$$a_{\gamma}\mathcal{O}=\{x\in K\mid v_p(x)\geq \gamma\}$$

We can define similarly : 
$$ Z_{\gamma,\delta}:=\left\{\begin{pmatrix}
a & b \\ b\delta & a
\end{pmatrix}\in SL_2(K) \mid b\in a_{\gamma}\mathcal{O}, a\in 1+a_{2\gamma}\delta\mathcal{O} \mbox{ and } a^2-b^2\delta=1\right\}$$

We have the same proposition :
\begin{thm}\label{L3}
If $K$ is a $p$-adically closed field (or a model of $Th(\Qp^{an})$). For $p\neq 2$ and $\delta \in\{p,\alpha p\}$ (resp. for $\delta=\alpha$).
\begin{enumerate}
	\item $Z_{0,\delta}$ (resp. $Z_{1,\delta})$ is of finite index in $Q_{\delta}$.
	\item The $Z_{\gamma,\delta}$ are the only one definable subgroups of $Z_{0,\delta}$ (resp. $Z_{1,\delta}$).
\end{enumerate}
\end{thm}

\begin{proof}
The property being true for $\Qp$, we reason by elementary equivalence. 
The $Z_{\gamma,\delta}$ are uniformly definable by $\varphi(x, a_{\gamma})$. As $Q_{\delta}$ is definable without parameters, for every formula $\psi(x,\bar{b})$ with parameters $\bar{b}$ in $K$, we have : 
$$\Qp\models \exists g_1, ... , g_n\in Q_{\delta} \forall x\in Q_{\delta} \quad \bigvee_{i=1}^n \varphi(g_i^{-1}x,1)$$
$$\Qp\models \forall \bar{b}\ \  "\psi(x,\bar{b}) \mbox{ defines a subgroup of } Q_{\delta}" \longrightarrow \exists a\forall x\   (\psi(x,\bar{b})\longleftrightarrow \varphi(x,a))$$
We can deduce the property for $K$.
\end{proof}

For $p=2$, we can similarly define : 
$$ Z_{\gamma,\delta}:=\left\{\begin{pmatrix}
a & b \\ b\delta & a
\end{pmatrix}\in SL_2(K) \mid b\in a_{\gamma}\mathcal{O}, a\in 1+a_{2\gamma-1}\delta\mathcal{O} \mbox{ and } a^2-b^2\delta=1\right\}$$
we have yet : 

\begin{thmbis}{L3}\label{L3bis}
If $K$ is a $p$-adically closed field (or a model of $Th(\mathbb{Q}_2^{an})$). For $p= 2$ and $v_p(\delta)=1$ (resp. for $v_p(\delta)=0$).
\begin{enumerate}
	\item $Z_{1,\delta}$ (resp. $Z_{2,\delta})$ is of finite index in $Q_{\delta}$.
	\item The $Z_{\gamma,\delta}$ are the only one definable subgroups of $Z_{1,\delta}$ (resp. $Z_{2,\delta}$).
\end{enumerate}
\end{thmbis}

\begin{rmk}
A natural question is to describe definable subgroups of $U$, $Q_1$ and $Q_{\delta}$, for $K$ a field elementarily equivalent to a finite extension of $\Qp$. This aim is more complicated and withaout real interest. For example, if $K$ is a totally ramificated extension of $\Qp$ of degree $n$, then definable subgroups of $K^+$ are of the form :
$$\sum_{i=0}^{n-1} \zeta_i\Zp+\pi^{\gamma} \OO$$
where $\OO$ is the valuation ring of $K$, $\gamma\in \Gamma$ and $\zeta\in K$.
\end{rmk}

We described the definable subgroups of $U$, $Q_1$ and $Q_{\delta}$ for $\delta\in K^{\times}\backslash (K^{\times})^2$, by Proposition \ref{solvable} and Fact \ref{subgroupsB}, we know what the definable subgroups of $B$ are. Then it give us a complete description of all definable nipotent or solvable subgroups of $SL_2(K)$. What about non solvable definable subgroups?

% Non solvable subgroups
\subsection{Non solvable definable subgroups}

For $x\in SL_2(K)$, we call valuation of $x$ the minimum $v(x)$ of the valuations of its coefficients. 
A subgroup $H$ of $SL_2(K)$ is said to be \textit{bounded} if there is some $m\in \Gamma$ such that : 
$$\forall x\in H \quad v(x)\geq m$$

\begin{rmk}
For $K=\Qp$, and for $H$ a definable subgroup of $SL_2(\Qp)$, \emph{bounded} means exactly \emph{compact}.
\end{rmk}

\begin{thm}\label{unbounded}
For $K$ a $p$-adically closed field and $H$ a definable subgroup of $SL_2(K)$. 
If $H$ is non-solvable and unbounded, then $H=SL_2(K)$.
\end{thm}

The theorem remains true for $K$ a henselian valued field of characteristic 0 whose value group is a $\mathbb{Z}$-group. For example, $K$ can be a finite extension of a $p$-adically closed field.  

\begin{proof}
$H$ is not solvable, so by proposition \ref{dimension}, we have $\dim H=3=\dim SL_2(K)$, and $H$ contains a neighborhood of the identity (Lemma \ref{lemdim}), then
$$\left(\begin{array}{cc}1+a_{\gamma_1}\mathcal{O} & a_{\gamma_2}\mathcal{O} \\a_{\gamma_3}\mathcal{O} & 1+a_{\gamma_4}\mathcal{O}\end{array}\right) \cap SL_2(K)\subseteq H$$
for some $a_{\gamma_i}\in K$ such that $v(a_{\gamma_i})=\gamma_i$. In particular :
$$\left(\begin{array}{cc}1 & a_{\gamma_2}\mathcal{O} \\0 & 1\end{array}\right)
\subseteq H$$
We note $Z=a_{\gamma_2}\mathcal{O}$.

$H$ is not bounded, so by Corollary \ref{union}, $H\cap Q_1$ or $H\cap U$ is non bounded. 
\begin{enumerate}
	\item If $H\cap Q_1$ is not bounded : we note $P$ the subgroup of $K^{\times}$ such that :
	$$H\cap Q_1=\left(\begin{array}{cc}P & 0 \\0 & P\end{array}\right)$$
	Let be $x\in K$ and $t\in P$ such that $v(t)< v(x a_{\gamma_2}^{-1})$. Then there is some $u\in \mathcal{O}$ such that $x=ta_{\gamma_2}u$. It follows that $P\cdot Z=K$. From 
	$$\left(\begin{array}{cc}t & 0 \\0 & t^{-1}\end{array}\right)\left(\begin{array}{cc}t^{-1} & u \\0 & t\end{array}\right)=\left(\begin{array}{cc}1 & tu \\0 & 1\end{array}\right)$$	
	we deduce that $U^{+}\subseteq H$, we can show the same for the transpose $^tU^+\subseteq H$. By 
	$$\left(\begin{array}{cc}1 & t \\0 & 1\end{array}\right)\left(\begin{array}{cc}1 & 0 \\-t^{-1} & 1\end{array}\right)\left(\begin{array}{cc}1 & t \\0 & 1\end{array}\right)=\left(\begin{array}{cc}0 & t \\-t^{-1} & 0\end{array}\right)    \quad (R1)  $$  $$\mbox{ and }\quad 
	\left(\begin{array}{cc}0 & t \\-t^{-1} & 0\end{array}\right)\left(\begin{array}{cc}0 & -1 \\1 & 0\end{array}\right)=\left(\begin{array}{cc}t & 0 \\0 & t^{-1}\end{array}\right) \quad(R2)$$
	we conclude that $w\in H$ and $Q_1\subseteq H$. Finally $B\subseteq H$ and by Bruhat decomposition $H=SL_2(K)$.
	
	\item If $H\cap U^+$ is  unbounded, then $U^+\subseteq H$ because every proper subgroup of $U^+$ is bounded. We also know that 
	$$\left(\begin{array}{cc}1 & 0 \\a_{\gamma_3}\mathcal{O} & 1\end{array}\right)\subseteq H$$
	If $\gamma_3\leq 0$, then by $(R1)$, $w\in H$. If not, take $t\in K$ such that $v(t)\geq \gamma_3$, then by $(R1)$ :
	$$\left(\begin{array}{cc}0 & -t^{-1} \\t & 0\end{array}\right), \left(\begin{array}{cc}0 & -t^{-2} \\t^2 & 0\end{array}\right)\in H \mbox{ and } \left(\begin{array}{cc}0 & -t^{-1} \\t & 0\end{array}\right)\left(\begin{array}{cc}0 & -t^{-2} \\t^2 & 0\end{array}\right)	=\left(\begin{array}{cc}-t & 0 \\0 & -t^{-1}\end{array}\right)\in H$$
	 In any case, $H$ contains an element of $Q_1$ of non-zero valuation. Because $\Gamma$ is a $\mathbb{Z}$-group, any definable subgroup of $\Gamma$ is either trivial or non bounded. Indeed, for $\varphi(x,\bar{a})$ a formula, we have 
	 $$ \mathbb{Z}\models \forall \bar{a}\  "\varphi(x,\bar{a}) \mbox{ defines a subgroup of }\mathbb{Z}" \longrightarrow \forall x \  [\varphi(x,\bar{a}) \rightarrow \exists y \ (y>x \wedge \varphi(y,\bar{a}))]$$
	  Since $v(P)$ is a non-trivial definable subgroup of $\Gamma$, it is unbounded and $H\cap Q_1$ is unbounded. We can conclude using the first case. 
\end{enumerate}
\end{proof}

\begin{rmk}
A consequence of the previous theorem is that, for $K$ a $p$-adically closed field, $SL_2(K)$ is definably connected (this means that it does not have a proper definable subgroup of finite index).
\end{rmk}

\begin{fact}[{\cite[Chap II, 1.3, Proposition 2]{Serrearbres}}]
For $K$ a $p$-adically closed field and $H$ a definable subgroup of $SL_2(K)$. 
If $H$ is bounded then $H$ is contained in a conjugate of $SL_2(\mathcal{O})$. 
\end{fact}

We denote :
$$H_{\gamma,\eta_1,\eta_2}=\left(\begin{array}{cc}1+a_{\gamma}\mathcal{O} & a_{\eta_1}\mathcal{O} \\a_{\eta_2}\mathcal{O} & 1+a_{\gamma}\mathcal{O}\end{array}\right)\cap SL_2(K) \quad \mbox{ with } v_p(a_{\gamma})=\gamma \mbox{ and } v_p(a_{\eta_i})=\eta_i$$
If $\eta_1+\eta_2\geq \gamma\geq 0$, then $H_{\gamma,\eta_1,\eta_2}$ is a subgroup of $SL_2(K)$, and it is a neighborhood of the identity. The groups of the form $H_{\gamma,\eta_1,\eta_2}$ are examples of definable non sovable bounded subgroups of $SL_2(K)$. The following proposition gives a sort of converse of this fact :

%$\bullet$ Let us remark that $\left(\begin{array}{cc}1+p^k\Zp & p^{n_1}\Zp \\p^{n_2}\Zp & 1+p^k\Zp\end{array}\right)$ with $n_1+n_2\geq k\geq 0$ is a subgroup of $GL_2(\Qp)$, it follow that $\left(\begin{array}{cc}1+p^k\Zp & p^{n_1}\Zp \\p^{n_2}\Zp & 1+p^k\Zp\end{array}\right)\cap SL_2(\Qp)$ is a subgroup of $SL_2(\Qp)$ and an open neighborhood de $I$. 

\begin{prop}\label{bounded}
Let $K$ be a $p$-adically closed field  and $H$ a definable subgroup of $SL_2(K)$. If $H$ is bounded and non solvable and if $w$ normalize $H$ then, up to conjugacy,
\begin{itemize}
	\item either, there exists $\gamma,\eta\in \Gamma$ and $a_{\gamma},a_{\eta}\in K$ with $v_p(a_{\gamma })=\gamma$ and $v_p(a_{\eta})=\eta$ such that $H_{\gamma, \eta, \eta}$ is subgroup of finite index of $H$ at most $2(p-1)$ if $p\neq 2$ (or at most $4$ if $p=2$), where : 
$$H_{\gamma,\eta,\eta}=\left(\begin{array}{cc}1+a_{\gamma}\mathcal{O} & a_{\eta}\mathcal{O} \\a_{\eta}\mathcal{O} & 1+a_{\gamma}\mathcal{O}\end{array}\right) \cap SL_2(K) \quad \mbox{ where } 2\eta\geq \gamma> 0$$ 
	\item or $H=SL_2(\mathcal{O})$.
\end{itemize}
\end{prop}

\begin{proof}
We reason in $\Qp$, up to conjugacy, we can assume that $H\leq SL_2(\Zp)$.
We denote $B_H=B\cap H$, we know by the Fact \ref{subgroupsB}, that :
$$B_H=\left(\begin{array}{cc}P & Z \\0 & P\end{array}\right) \quad \mbox{ with } Z=p^n\Zp \mbox{ and } 1+p^k\Zp\leq P \mbox{ of finite index at most }(p-1)$$
for some $n\geq 0$ and $k\geq 1$. 

$\bullet$ If $2n\geq k$, then there exists a subgroup $H'$ such that $B_H\leq H'\leq H$ and $H'$ of the form $\left(\begin{array}{cc}P & Z \\Z & P\end{array}\right)$ (it suffices to take $H'=B_H\cdot V$ where $ V$ is a neighborhood of the indentity in $SL_2(\Qp)$).
Replacing $H'$ by $\pm H'$, we may assume that $-I\in H'$. Since $H'^w=H'$, $H'\cup wH'$ is a subgroup containing $H'$ as a subgroup of index 2. Similarly, we can can assume that $-I\in H$ and $H$ of index at most 2 in the subgroup $H\cup wH$. By Bruhat decomposition, $H\cup wH=B_H \cup B_H w B_H$. Since $B_H\subseteq H'$ and $w\in H'\cup wH'$, $H\cup wH=H'\cup wH'$, then $H'$ is a subgroup of index at most 2 in $H$.  For $$\left(\begin{array}{cc}x(1+p^ka) & p^nb \\p^nc & x'(1+p^kd)\end{array}\right)\in H',\ (\mbox{where }x,x' \mbox{ are } p^{th} \mbox{ roots of the unity and } a,b,c,d\in\Zp )$$  because the determinant is 1, we see that $xx'=1$, and so 
$$\left(\begin{array}{cc}x(1+p^ka) & p^nb \\p^nc & x'(1+p^kd)\end{array}\right)=\left(\begin{array}{cc}x & 0 \\0 & x'\end{array}\right)\left(\begin{array}{cc}1+p^ka & x^{-1}p^nb \\x'^{-1}p^nc' & 1+p^kd\end{array}\right)$$
In other terms, $H_{k,n,n}$ is subgroup  of finite index at most $(p-1)$ in $H'$ so at most $2(p-1)$ in $H$.

$\bullet$ If $2n<k$ with $n>0$, then $B_H$ and $B_H^w$ generate $H'=\left(\begin{array}{cc}P' & Z \\Z & P'\end{array}\right)$ with $1+p^{2n}\Zp\leq P'$ of finite index at most $(p-1)$ and $P<P'$. This is a contradiction, because $H'\cap B$ must be contained in $B_H$. 

$\bullet$ If $2n<k$ and $n=0$, then $\left(\begin{array}{cc}1 & \Zp \\0 & 1\end{array}\right)\subseteq H$, and by the action of $w$, $\left(\begin{array}{cc}1 & 0 \\\Zp & 1\end{array}\right)\subseteq H$. 
By $(R1)$ and $(R2)$ we can deduce that $w\in H$, $\left(\begin{array}{cc}\Zp^{\times} & 0 \\0 & \Zp^{\times}\end{array}\right)\subseteq H$ and $\left(\begin{array}{cc}\Zp^{\times} & \Zp \\0 & \Zp^{\times}\end{array}\right)\subseteq H$. By Bruhat decomposition on $SL_2(\Zp)$, we have $H=SL_2(\Zp)$.

By elementarily equivalence we can now easily deduce the proposition for all $p$-adically closed fields $K$ :  for $\varphi(\bar{x},\bar{a})$ a formula, as $H_{\gamma,\eta,\eta}$ is definable by $\psi(\bar{x},\bar{b})$ 
\begin{align}
K\models \forall\bar{a} \ &\mbox{"}\varphi(\bar{x},\bar{a}) \mbox{ defines a bounded subgroup of } SL_2(K)\mbox{ of dimension  3 and normalised by }w\mbox{"} \nonumber\\
& \longrightarrow  \left[\exists \bar{b}\ \exists \bar{x}_1, ... , \bar{x}_{2(p-1)} \bigwedge_{0\leq i\leq 2(p-1)} \varphi(\bar{x}_i,\bar{a})\ \wedge\  \forall \bar{y} \ (\varphi(\bar{y},\bar{a})\rightarrow \bigvee_{0\leq i\leq2(p-1)}\psi(\bar{y}\cdot \bar{x}_i^{-1},\bar{b}))\right] \nonumber
\end{align}

%$\bullet$ We suppose $w\in H$, then $w^2=-I\in H$. We put $B_H=B\cap H=\left(\begin{array}{cc}P & Z \\0 & P\end{array}\right)$, where $Z$ is a definable subgroup of $\Qp^+$ so $Z=p^n\Zp$ for some $n\in \mathbb{Z}$, and $P$ is a subgroup of $\Qp^{\times}$ then $P$ contains $1+p^k\Zp$ as a definable subgroup of finte at most $(p-1)$, more $\pm I\in P$. We have $P\cdot Z\subseteq Z$ then $k\geq 0$.

\end{proof}

\begin{conj}
For $K$ a $p$-adically closed field  and $H$ a definable subgroup of $SL_2(K)$. If $H$ is bounded and non solvable and if $w$  does not normalize $H$ then, up to conjugacy, there exists $\gamma,\eta_1,\eta_2\in \Gamma$ such that $H_{\gamma, \eta, \eta}$ is subgroup of finite index of $H$ at most $2(p-1)$ if $p\neq 2$ (or at most $4$ if $p=2$).
\end{conj}

\begin{rmk}
By the remark on page \pageref{Qpan}, we can transfer theorem \ref{unbounded} and proposition \ref{bounded} to $\Qp^{an}$, if we replace "non solvable" by "of dimension 3".
\end{rmk}

The following tabular sum up the description of all definable subgroups of $SL_2(K)$ up to conjugacy for $K$ a $p$-adically closed field : \label{tabular}
\\ \begin{tabular}{|c|c|p{7.5cm}|}
\hline
Algebraic properties & "Frame" subgroups & \centering Definable subgroups in $p$-adically closed field \tabularnewline
\hline
\multirow{3}*{Nilpotent }&$U$ and $U_+$ & \centering $\left(\begin{array}{cc}1 & a_{\gamma} \mathcal{O} \\0 & 1 \end{array}\right)$ or $\left(\begin{array}{cc}\pm 1 & a_{\gamma} \mathcal{O} \\0 & \pm 1\end{array}\right)$\tabularnewline
\cline{2-3}
 & $Q_1$ &  \centering virtually  $\left(\begin{array}{cc}1+a_{\gamma}\mathcal{O} & 0 \\0 & 1+a_{\gamma}\mathcal{O}\end{array}\right)$ or  $\left(\begin{array}{cc}\{b_{\gamma}\}_{\gamma\in n\Gamma}\cdot(1+a_{\gamma}\mathcal{O}) & 0 \\0 & \{b_{\gamma}\}_{\gamma\in n\Gamma}\cdot(1+a_{\gamma}\mathcal{O})\end{array}\right)$\tabularnewline
 \cline{2-3}
 & $Q_{\delta}$ & \centering $Z_{\gamma,\delta}$ for $\gamma\in \Gamma^{>0}$ \tabularnewline
\hline
\multirow{2}*{Solvable, non nilpotent} & $N_G(Q_1)$& \\
\cline{2-3}
& $B$ & \centering $\left(\begin{array}{cc}P & Z \\0 & P\end{array}\right)$ with $P\leq K^{\times}$, $Z\leq K^+$ and $P\cdot Z\subseteq Z$\tabularnewline
\hline
Non solvable, compact &  & \centering contained in $SL_2(\mathcal{O})^g$ ($g\in SL_2(K)$)\tabularnewline
\cline{2-3}
Non solvable, non compact &   & \centering $ SL_2(K)$\tabularnewline
\hline
\end{tabular}

%GENEROSITY 
\section{Generosity of the Cartan subgroups}
%Our purpose is now to show the generosity of the Cartan subgroup $Q_1$. It follows from the next more general proposition :
The important notion of genericity was particularly developped by Poizat for groups in stables theories \cite{Poizat_stable}. In any group that admits a geometric dimension notion, one expects to characterise genericity in terms of maximal dimension. 
%For a group of finite Morley rank, a definable subset is generic if and only if it has maximal dimension. 
The term generous was introduced by Jaligot in \cite{Jal_never} to show some conjuguation theorem. The aim of this section is to say which definable subgroups are genereous. We are showing that for $K$ a $p$-adically closed field the only one generous Cartan subgroup of $SL_2(K)$ is $Q_1$ up to conjugacy. That generalize the same result for real closed fields shown in \cite{cartan}. 

\begin{defi}
\begin{itemize}
	\item A part $X$ in a group $G$ is said generic if $G$ can be covered by finitely many translates of $X$ : 
	$$G=\bigcup_{i=1}^n g_i\cdot X$$
	\item $X$ is generous if the union of its conjugate $X^G=\bigcup_{g\in G}X^g$ is generic.
\end{itemize}
\end{defi}

\begin{rmk}
If $G$ is a definable group in a $p$-adically closed field, then genericity imply of being of maximal dimension. The converse is false : $\dim \Zp=\dim \Qp$ but $\Zp$ is not generic in $(\Qp,+)$. 
\end{rmk}

Let us begin by a general proposition true for every valued field.
\begin{prop}
Let $(K,v)$ be valued field.
\begin{enumerate}
  \item The set $W=\{A\in SL_2(K)\mid v(tr(A))<0\}$ is generic in $SL_2(K)$.
  \item The set $W'=\{A\in SL_2(K) \mid v(tr(A))\geq 0\}$ is not generic in $SL_2(K)$.\end{enumerate}
\end{prop}

\begin{proof}
1. 
We consider the matrices : 
$$A_1 = I ,\quad A_2= \left(\begin{array}{cc}0 & 1 \\-1 & 0\end{array}\right), \quad A_3=\left(\begin{array}{cc}a^{-1} & 0 \\0 & a\end{array}\right) \quad \mbox{ and } \quad A_4=\left(\begin{array}{cc}0 & -b^{-1} \\b & 0\end{array}\right)$$
with $v(a)>0$ and $v(b)>0$.

We show that $SL_2(K)=\bigcup_{i=1}^{4} A_iW$.
Suppose there exists  $$M=\left(\begin{array}{cc}x & y \\u & t\end{array}\right)\in SL_2(K)$$ such that $M\nin \bigcup_{i=1}^{4} A_iW$.

Since $M\nin A_1W\bigcup A_2W$, we have $x+t =\varepsilon$ and $y-u=\delta$ with $v(\varepsilon)\geq 0$ and $v(\delta)\geq 0$. 
Since $M\nin A_3W$, we have $ax+a^{-1}t=\eta$ with $v(\eta)\geq 0$.
We deduce $t=\frac{\eta-a\varepsilon}{a^{-1}-a}$.
Similarly, it follows from $M\nin A_4W$ that $u=\frac{\theta-b\delta}{b^{-1}-b}$ with some $\theta$ such that $v(\theta)\geq0$. 

Since $v(a)>0$, we have $v(a+a^{-1})<0$. 
From $v(\eta - a\varepsilon)\geq \min\{v(\eta);v(a\varepsilon)\}\geq0$,
we deduce that $v(t)=v(\frac{\eta-a\varepsilon}{a+a^{-1}})=v(\eta-a\varepsilon)-v(a+a^{-1})>0$. 
Similarly $v(u)>0$. 
It follows that $v(x)=v(\varepsilon -t)\geq 0$ and $v(y)\geq0$.

Therefore $v(det(M))=v(xt-uy)\geq \min\{v(xt),v(uy)\}>0$ and thus $det(M)\neq 1$, a contradiction .

2.  
We show that the family of matrices $(M_x)_{x\in K^{\times}}$ cannot be covered by finitely many $SL_2(K)$-translates of $W'$, where :
$$M_x =\left(\begin{array}{cc}x & 0 \\0 & x^{-1}\end{array}\right)$$
Let  $A=\left(\begin{array}{cc}a & b \\c & d\end{array}\right) \in SL_2(K)$. 
Then $tr(A^{-1}M_x)=dx+ax^{-1}$. If $v(x)>\max\{|v(a)|, |v(d)|\}$ then $v(tr(A^{-1}M_x))<0$ and $M_x\nin AW'$. 

Therefore for every finite family $\{A_j\}_{i\leq n}$, there exists $x\in K$ such that $M_x\nin \bigcup_{j=1}^{n} A_jW'$.
\end{proof}

\begin{rmk}
We remark that the sets $W$ and $W'$ form a partition of $SL_2(K)$. They are both definable in the field language if the valuation $v$ is definable in $K$.
\end{rmk}

We focus now on $\Qp$, with the notation from previous section, we define the angular component $ac:\Qp\longrightarrow \Fp$ by $ac(x)=res(p^{-v_p(x)}x)$. Thus, if $p\neq 2$ an element $x\in\Qp^{\times}$ is a square if and only if $v_p(x)$ is even and $ac(x)$ is a square in $\Fp$. For $p=2$, an element $x\in \mathbb{Q}_2$ can be written $x=2^nu$ with $n\in \Z$ and $u\in\Z_2^{\times}$, then $x$ is a square if $n$ is even and $u\equiv 1 \mod 8$ \cite{JPS}.

\begin{lem}\label{W}
$W\subseteq Q_1^{SL_2(\Qp)}$ and for $\delta\in \Qp^{\times}\backslash (\Qp^{\times})^2$ and $\mu\in GL_2(\Qp)$, $Q_{\delta}^{\mu \cdot SL_2(\Qp)}\subseteq W'$, moreover $U^{SL_2(\Qp)}\subseteq W'$.
\end{lem}

\begin{proof}
Let be $A\in SL_2(\Qp)$ with $v_p(tr(A))<0$.
\\For $p\neq 2$, since $v_p(tr(A))<0$,  $v_p(tr(A)^2-4)=2v_p(tr(A))$ and $ac(tr(A)^2-4)=ac(tr(A)^2)$, so $tr(A)^2-4$ is a square in $\Qp$. 
\\For $p=2$, we can write $tr(A)=2^nu$ with $n\in \Z$ and $u\in \Zp^{\times}$. Then $tr(A)^2-4=2^{2n}(u^2-4\cdot 2^{-2n})$. Since $n\leq -1$,  $u^2-4\cdot 2^{-2n} \equiv u^2 \equiv 1 (\mbox{mod } 8)$, so $tr(A)^2-4\in (\mathbb{Q}_2^{\times})^2$.

In all cases, by the proposition \ref{tr}, $W\subseteq Q_1^{SL_2(\Qp)}$ and, by complementarity, $Q_{\delta}^{\mu\cdot SL_2(\Qp)}\subseteq W'$, and $U^{SL_2(\Qp)}\subseteq W'$.
\end{proof}

We can now conclude with the following corollary, similar to \cite[Remark 9.8]{cartan}:

\begin{cor} \label{generous} Let $K$ be a $p$-adically closed field.
\begin{enumerate}
  \item The Cartan subgroup $Q_1$ is generous in $SL_2(K)$.
  \item The Cartan subgroups $Q_{\delta}^{\mu}$ (for $\delta\in K^{\times}\backslash (K^{\times})^2$ and $\mu\in GL_2(K)$) are not generous in $SL_2(K)$. 
  \item $U$ is not generous.
\end{enumerate}
\end{cor}

\begin{proof}
Lemma \ref{W} shows that $Q_1^{SL_2(\Qp)}$ is generic. $Q_1^{SL_2(\Qp)}$ is definable by free parameter formula $\varphi(x)$, so 
$$\Qp \models \exists a_1, ... , a_n \in S \quad \forall x\in S\quad \bigvee _{i=1}^n \varphi(a_i^{-1}x)$$
$K$ satisfies the same formula and $Q_1$ is generous in $SL_2(K)$. We reason similarly for $Q_{\delta}$ and $U$.
\end{proof}

\begin{rmk}
If $K$ is a field elementary equivalent to a finite extension of $\Qp$, there is the same caracterisation of square in $K^{\times}$, so the same result than Corollary \ref{generous} is true.
\end{rmk}

%\nocite{*}
\bibliographystyle{plain}
\bibliography{truc}

\begin{thebibliography}{10}

\bibitem{cartan}
Elias Baro, Eric Jaligot, and Margarita Otero.
\newblock Cartan subgroups of groups definable in o-minimal structures.
\newblock arXiv:1109.4349v2 [math.GR], 2011.

\bibitem{borel}
Armand Borel.
\newblock {\em Linear Algebraic Groups}.
\newblock Graduate Texts in Mathematics. Springer, second enlarged edition
  edition, 1991.

\bibitem{cherlin}
Gregory Cherlin.
\newblock Groups of small morley rank.
\newblock {\em Annals of Mathematical Logic}, 17:1--28, 1979.

\bibitem{chevalley}
Claude Chevalley.
\newblock {\em Th{\'e}orie des groupes de Lie II. Groupes alg{\'e}briques}.
\newblock Hermann, Paris, 1951.

\bibitem{Frec}
Olivier Fr{\'e}con.
\newblock Conjugacy of carter subgroups in groups of finite morley rank.
\newblock {\em Journal of Mathematical Logic}, 8(1):41--92, 2008.

\bibitem{FrecJal}
Olivier Fr{\'e}con and Eric Jaligot.
\newblock The existence of carter subgroups in groups of finite morley rank.
\newblock {\em Journal of Group Theory}, 8:623--633, 2005.

\bibitem{SL2R}
Jakub Gismatullin, Davide Penazzi, and Anand Pillay.
\newblock Some model theory of $sl(2,\mathbb{R})$.
\newblock arXiv : 1208.0196, 2012.

\bibitem{hodges}
Wilfrid Hodges.
\newblock {\em Model Theory}, volume~42 of {\em Encyclopedia of Mathematics and
  its Applications}.
\newblock Cambridge University Press, 1993.

\bibitem{humphreys}
James~E. Humphreys.
\newblock {\em Linear Algebraic Groups}.
\newblock Graduate Texts in Mathematics. Springer, 1998.

\bibitem{Jal_never}
Eric Jaligot.
\newblock Generix never gives up.
\newblock {\em J. Symbolic Logic}, 71(2):599--610, 2006.

\bibitem{MacEQ}
Angus MacIntyre.
\newblock On definable subsets of $p$-adic fields.
\newblock {\em The Journal of Symbolic Logic}, 41(3):605--610, 1976.

\bibitem{Marker}
David Marker.
\newblock {\em Model Theory : An Introduction}, volume 217 of {\em Graduate
  Texts in Mathematics}.
\newblock Springer, 2002.

\bibitem{Samaria}
Samaria Montenegro.
\newblock Pseudo real field, pseudo $p$-adic closed fields and $ntp_2$.
\newblock arXiv:1411.7654, 2014.

\bibitem{pillayfields}
Anand Pillay.
\newblock On fields definable in $\mathbb{Q}_p$.
\newblock {\em Archive for Mathematical Logic}, 29:1--7, 1989.

\bibitem{Poizat_stable}
Bruno Poizat.
\newblock {\em Groupes stables}.
\newblock Nur al-Mantiq wal-Ma'rifah [Light of Logic and Knowledge], 2. Bruno
  Poizat, Lyon, 1987.
\newblock Une tentative de conciliation entre la g{\'e}om{\'e}trie
  alg{\'e}brique et la logique math{\'e}matique. [An attempt at reconciling
  algebraic geometry and mathematical logic].

\bibitem{JPS}
Jean-Pierre Serre.
\newblock {\em Cours d'arithm{\'e}tique}.
\newblock puf, 1970.

\bibitem{Serrearbres}
Jean-Pierre Serre.
\newblock {\em Arbres, amalgames, $SL_2$}.
\newblock Number~46 in ast{\'e}risque. Soci{\'e}t{\'e} Math{\'e}matqiue de
  France, third edition, 1983.

\bibitem{Suzuki}
Michio Suzuki.
\newblock {\em Group Theory II}, volume 248 of {\em Grundlehren der
  matematischen Wissenschaften}.
\newblock Springer, 1986.

\bibitem{dimvdD}
Lou van~den Dries.
\newblock Dimension of definable sets, algebraic boundness and henselian
  fields.
\newblock {\em Annals of Pure and Applied Logic}, 45:189--209, 1989.

\bibitem{vdDHasMac}
Lou van~den Dries, Dierdre Haskell, and Dugald Macpherson.
\newblock One-dimensional $p$-adic subanalytic sets.
\newblock {\em Journal of London Mathematical Society}, 49:1--20, 1999.

\bibitem{Dries}
Lou van~den Dries and Philip Scowcroft.
\newblock On the structure of semialgebraic sets over $p$-adic fields.
\newblock {\em The Journal of Symbolic Logic}, 53(4):1138--1164, 1988.

\end{thebibliography}

\end{document}